\documentclass[11pt,reqno]{amsart}

\usepackage{amsthm,amssymb}
\usepackage[pagebackref,colorlinks,linkcolor=red,citecolor=blue,urlcolor=blue,hypertexnames=false]{hyperref}
\usepackage{amsrefs}
\usepackage{amscd}

\input{xy}
\xyoption{all}
\newdir{ >}{{}*!/-5pt/@{>}}

\numberwithin{equation}{section}

\theoremstyle{plain}
\newtheorem{thm}[equation]{Theorem}
\newtheorem{prop}[equation]{Proposition}
\newtheorem{coro}[equation]{Corollary}
\newtheorem{lem}[equation]{Lemma}

\theoremstyle{definition}
\newtheorem{defi}[equation]{Definition}

\theoremstyle{remark}
\newtheorem{nota}[equation]{Notation}
\newtheorem{rem}[equation]{Remark}

\newtheorem*{ack}{Acknowledgements}
\theoremstyle{definition}
\newtheorem{exa}[equation]{Example}
\newtheorem{exas}[equation]{Examples}

\newcommand{\triqui}{\vartriangleleft}
\newcommand{\troqui}{\vartriangleright}

\renewcommand{\top}{\mathrm{top}}

\newcommand{\tor}{\mathrm{Tor}}
\newcommand*{\colim}{\mathop{\mathrm{colim}}}
\newcommand{\weq}{\overset\sim\longrightarrow}
\newcommand{\iso}{\overset\cong\longrightarrow}

\newcommand{\bpf}{\begin{proof}}
\newcommand{\epf}{\end{proof}}
\newcommand{\bprop}{\begin{prop}}
\newcommand{\eprop}{\end{prop}}
\newcommand{\bthm}{\begin{thm}}
\newcommand{\ethm}{\end{thm}}
\newcommand{\brk}{\begin{rem}}
\newcommand{\erk}{\end{rem}}
\newcommand{\bdefi}{\begin{defi}}
\newcommand{\edefi}{\end{defi}}
\newcommand{\blemma}{\begin{lem}}
\newcommand{\elemma}{\end{lem}}
\newcommand{\bclly}{\begin{coro}}
\newcommand{\eclly}{\end{coro}}
\newcommand{\bnota}{\begin{nota}}
\newcommand{\enota}{\end{nota}}

\newcommand{\be}{\begin{enumerate}}
\newcommand{\ee}{\end{enumerate}}
\newcommand{\topdf}{\texorpdfstring}

\newcommand{\comment}[1]{}
\newcommand{\ab}{\mathfrak{Ab}}

\newcommand{\elli}{\ell^\infty}
\newcommand{\ellp}{\ell^p}
\newcommand{\ellq}{\ell^q}
\newcommand{\lf}{\lfloor}
\newcommand{\rf}{\rfloor}
\newcommand{\Ell}{\mathcal{L}}
\newcommand{\Ellp}{\mathcal{L}^p}

\newcommand{\cA}{\mathcal{A}}
\newcommand{\cB}{\mathcal{B}}
\newcommand{\cC}{\mathcal{C}}

\newcommand{\cE}{\mathcal{E}}

\newcommand{\cK}{\mathcal{K}}

\newcommand{\cF}{\mathcal{F}}

\newcommand{\ran}{\mbox{ran}}

\newcommand{\dom}{\mbox{dom}}

\newcommand{\diag}{\mathrm{diag}}
\newcommand{\emb}{\mathrm{Emb}}
\newcommand{\estar}{\mathcal{E}^*}
\newcommand{\aij}{A_{ij}}

\newcommand{\Gami}{\Gamma^\infty}

\newcommand{\im}{\mbox{Im}}

\newcommand{\mspan}{\mbox{span}}

\newcommand{\hotimes}{\hat{\otimes}}
\newcommand{\sotimes}{\overset{\sim}{\otimes}}
\newcommand{\RR}{\mathbb{R}}
\newcommand{\R}{\RR}
\newcommand{\V}{\mathbb{V}}
\newcommand{\CC}{\mathbb{C}}
\newcommand{\C}{\CC}
\newcommand{\Q}{\mathbb{Q}}

\newcommand{\NN}{\mathbb{N}}
\newcommand{\N}{\NN}

\newcommand{\Z}{\mathbb{Z}}
\newcommand{\bH}{\mathbb{H}}
\newcommand{\supp}{\mbox{supp}}

\newcommand{\cP}{\mathcal{P}}

\newcommand{\fH}{\Gamma}
\newcommand{\fA}{\mathfrak{A}}
\newcommand{\fB}{\mathfrak{B}}

\newcommand{\fS}{\mathfrak{S}}

\title[Cyclic homology, crossed products, and stabilizations]{Cyclic homology, tight crossed products, and small stabilizations}

\author{Guillermo Corti\~nas}
\address{Dep. Matem\'atica-IMAS, FCEyN-UBA\\ Ciudad Universitaria Pab 1\\
1428 Buenos Aires\\ Argentina}
\email{gcorti@dm.uba.ar}\urladdr{http://mate.dm.uba.ar/\~{}gcorti}

\thanks{Research supported by CONICET and
by grants UBACyT W386, PIP 112-200801-00900, MTM2007-64704 (FEDER funds) and by MathAmSud network U11MATH-05. The latter network was partially funded by ANII, Uruguay, and by MINCyT, Argentina.}

\begin{document}

\begin{abstract}
In \verb|arXiv:1212.5901| we associated an algebra $\Gami(\fA)$ to every bornological algebra $\fA$ and an ideal $I_{S(\fA)}\triqui\Gami(\fA)$ to every
symmetric ideal $S\triqui\elli$. We showed that $I_{S(\fA)}$ has $K$-theoretical properties which are similar to those of the usual stabilization with respect to the ideal $J_S\triqui\cB$ of the algebra $\cB$ of bounded operators in Hilbert space which corresponds to $S$ under Calkin's correspondence. In the current article we compute the relative
cyclic homology $HC_*(\Gami(\fA):I_{S(\fA)})$. Using these calculations, and the results of \emph{loc. cit.}, we  
prove that if $\fA$ is a $C^*$-algebra and $c_0$ the symmetric ideal of sequences vanishing at infinity, then $K_*(I_{c_0(\fA)})$ is homotopy invariant, and that if $*\ge 0$, it contains $K^{\top}_*(\fA)$ as a direct summand. This is a weak analogue of the Suslin-Wodzicki theorem (\cite{sw1}) that says that for the ideal $\cK=J_{c_0}$ of compact operators and the $C^*$-algebra tensor product $\fA\sotimes\cK$, we have $K_*(\fA\sotimes\cK)=K^{\top}_*(\fA)$. Similarly, we prove that if $\fA$ is a unital Banach algebra
and $\ell^{\infty-}=\bigcup_{q<\infty}\ell^q$, then $K_*(I_{\ell^{\infty-}(\fA)})$ is invariant under H\"older continuous homotopies, and that for $*\ge 0$ it contains $K^{\top}_*(\fA)$
as a direct summand. These $K$-theoretic results are obtained from cyclic homology computations. We also compute the relative cyclic homology groups $HC_*(\Gami(\fA):I_{S(\fA)})$ in terms of $HC_*(\elli(\fA):S(\fA))$ for general $\fA$ and $S$. For $\fA=\C$ and general $S$, we further compute the latter
groups in terms of algebraic differential forms. We prove that the map $HC_n(\Gami(\C):I_{S(\C)})\to HC_n(\cB:J_S)$ is an isomorphism in many cases. 
\end{abstract}

\maketitle

\section{Introduction} 
Let $\ell^2=\ell^2(\N)$ be the Hilbert space of square-summable
sequences of complex numbers and $\cB=\cB(\ell^2)$ the algebra of
bounded operators. Calkin's theorem in \cite{calk}*{Theorem 1.6}, as restated by Garling in \cite{garling}*{Theorem ~1}, establishes an isomorphism
\[
S\mapsto J_S
\]
between the lattice of proper symmetric ideals of the algebra $\elli$ of bounded sequences and that of proper two-sided ideals of the algebra
$\cB=\cB(\ell^2)$ of bounded operators. In \cite{hl} we introduced a subalgebra $\Gami\subset\cB$ and showed that the above lattices are also isomorphic
to the lattice of proper two-sided ideals of $\Gami$, via the correspondence
\[
S\mapsto I_S=J_S\cap\Gami.
\]

More generally, we associated to each bornological algebra $\fA$, an algebra $\Gami(\fA)$ which contains an ideal $I_{S(\fA)}$ for each symmetric ideal
$S\triqui\elli$. We showed that the algebra $I_{S(\fA)}$ has $K$-theoretical properties which are analogous to those of the usual stabilization with respect to $J_S$, at least when $S$ is one of the following:
\begin{equation}\label{eq:lists}
S\in\{c_0, \ell^{p-},\ellq, \ell^{q+} \quad(p\le\infty, q<\infty)\}.
\end{equation}
Here $c_0$ is the ideal of sequences vanishing at infinity,
$\ellq$ consists of the $q$-summable sequences, and
\[
\ell^{p-}=\bigcup_{r<p}\ell^r,\ \ \ell^{q+}=\bigcap_{s>q}\ell^s.
\]
We proved that for $S$ as in \eqref{eq:lists}, there is a long exact sequence:
\begin{equation}\label{intro:seqkth}
\xymatrix{
KH_{n+1}(I_{S(\fA)})\ar[r]&HC_{n-1}(\Gami(\fA):I_{S(\fA)})\ar[d]\\
KH_{n}(I_{S(\fA)})&K_n(\Gami(\fA):I_{S(\fA)})\ar[l]
}
\end{equation}
If furthermore, $S\neq c_0$, then $KH_*(I_{S(\fA)})=KH_*(I_{\ell^1(\fA)})$. We proved that the functor $KH_*(I_{c_0(\fA)})$ is invariant under arbitrary continuous homotopies of bornological algebras, and that $KH_*(I_{\ell^1(\fA)})$ is invariant under H\"older continuous homotopies. We also showed that if $*\ge 0$ and either $\fA$ is a $C^*$-algebra and  $S=c_0$ or $\fA$ is a local Banach algebra and $S=\ell^1$, then $KH_*(I_{S(\fA)})$ contains $K^{\top}_*(\fA)$ as a direct summand.
In the current article we study the groups $HC_*(\Gami(\fA):I_{S(\fA)})$ for general $S$ and $\fA$. We show for example that if $\fA$ is a $C^*$-algebra
then $I_{c_0(\fA)}$ is $H$-unital and 
\[
HC_*(\Gami(\fA):I_{c_0(\fA)})=0.
\]
It follows from this, excision, and the exact sequence \eqref{intro:seqkth}, that the comparison map
\begin{equation}\label{intro:compa1}
K_*(I_{c_0(\fA)})\to KH_*(I_{c_0(\fA)})
\end{equation}
is an isomorphism. In particular, if $\fA$ is a $C^*$-algebra, then $K_*(I_{c_0(\fA)})$ is homotopy invariant, and if $*\ge 0$, it contains $K^{\top}_*(\fA)$ as a direct
summand. This again shows that $I_{c_0(-)}$ has properties analogous to those of $J_{c_0}=\cK$, the ideal of compact operators. Indeed, the result above is a weak analogue of the Suslin-Wodzicki theorem (Karoubi's conjecture) which says that if $\fA$ is a $C^*$-algebra then $K_*(\fA\sotimes\cK)=K_*^{\top}(\fA)$. We also show that if $\fA$ is a unital Banach algebra then $I_{\ell^{\infty-}(\fA)}$ is $H$-unital and 
\[
HC_*(\Gami(\fA):I_{\ell^{\infty-}(\fA)})=0.
\]
Thus the comparison map
\begin{equation}\label{intro:compa2}
K_*(I_{\ell^{\infty-}(\fA)})\to KH_*(I_{\ell^{\infty-}(\fA)})
\end{equation}
is an isomorphism. Again this is analogous to a similar property of stabilization with respect to $J_{\ell^{\infty-}}=\bigcup_p\Ellp$, the union of all Schatten ideals (see \cite{wodk}*{pp 490},\cite{cot}*{Theorem 8.2.5}).
In \cite{wodk}, M. Wodzicki studied the relative cyclic homology groups $HC_n(\cB:J_S)$. 
For $S$ as in \eqref{eq:lists}, the following integer was computed by Wodzicki in
\cite{wodk}*{Corollary to Theorem 8}
\[
m=m_S=\min\{n:HC_n(\cB:J_S)\ne 0\}.
\]
We prove in Proposition \ref{prop:hcmismo} that
\begin{equation}\label{intro:min}
m=\min\{n:HC_n(\Gami:I_S)\ne 0\},
\end{equation}
and that the natural map is an isomorphism for $n=m$:
\begin{equation}\label{intro:min2}
HC_m(\Gami:I_S)\iso HC_m(\cB:J_S).
\end{equation}
The techniques used in this article to establish the results above about $HC_*(\Gami(\fA):I_{S(\fA)})$ are similar to those used in \cite{wodk} to study the relative cyclic homology of stabilizations by $J_S$. We also obtain more results about the groups $HC_*(\Gami(\fA):I_{S(\fA)})$ using a different technique, which involves a description of $\Gami$ and $I_{S}$ as crossed products, established in \cite{hl}*{Proposition 6.12}. The inverse monoid $\emb$ of all partially defined
injections
\[
\N\supset\dom f\overset{f}{\longrightarrow}\N.
\]
acts on $\elli(\fA)$ by
\begin{equation}\label{intro:action}
f_*(\alpha)_n=\left\{\begin{matrix}\alpha_m& \text{ if }f(m)=n\\ 0&\text{ else.}\end{matrix}\right.
\end{equation}
By definition, an ideal $S\triqui\elli$ is symmetric if the action above maps $S$ to itself. 
Observe that if $A,B\subset\N$
are disjoint then the inclusions $p_A:A\to \N$ and $p_B:B\to \N$ satisfy
\[
(p_{A\cup B})_*=(p_{A})_*+(p_{B})_*
\]
In other words, the action above is \emph{tight} in the sense of Exel \cite{ruy}. 
Thus $\elli(\fA)$ is a module over the ring
\[
\Gamma=\Z[\emb]/\langle p_A+p_B-p_{A\cup B}: A\cap B=\emptyset\rangle
\]
Let $\cP\subset\Gamma$ be the subring generated by all the $p_A$ with $A\subset\N$. Note that $\cP$ is isomorphic to the subring of $\elli(\fA)$ consisting of those
sequences $\alpha:\N\to\Z$ which take finitely many distinct values. In particular \eqref{intro:action} makes $\cP$ into a $\Gamma$-module. Moreover $\elli(\fA)$ is a $\cP$-algebra, and the map
\begin{equation}\label{intro:quism}
HC(\elli(\fA):S(\fA))\to HC((\elli(\fA)/\cP:S(\fA))/\cP) 
\end{equation}
is a quasi-isomorphism (see Example \ref{exa:hhp} and \eqref{map:quisintro}). Furthermore the action of 
$\emb$ on $\elli(\fA)$ extends to a tight action on $HC(\elli(\fA):I_{S(\fA)})$, and we show that 
\begin{equation}\label{intro:hcbh}
HC_*(\Gami(\fA):I_{S(\fA)})=\bH_*(\Gamma/\cP:HC((\elli(\fA):S(\fA))/\cP)).
\end{equation}
Here the hyperhomology groups $\bH_*(\Gamma/\cP,-)$ are the hyperderived functors of the functor
\[
\Gamma-\mathrm{Mod}\to \ab,\ \ M\mapsto H_0(\Gami/\cP,M):=M\otimes_\Gamma\cP.
\]
We show in Proposition \ref{prop:ce=estar} that
\begin{multline}\label{intro:ce=estar}
H_0(\Gamma/\cP,M)=
M_\cE=\\
M/\mspan\{m-f_*(m):m\in M, f\in\emb\text{ such that } \dom f=\N\}.
\end{multline}
It follows from \eqref{intro:quism} and \eqref{intro:hcbh} that there is a first quadrant spectral sequence
\[
E^2_{p,q}=H_p(\Gamma/\cP,HC_q(\elli(\fA):S(\fA)))\Rightarrow HC_{p+q}(\Gami(\fA):I_{S(\fA)}).
\]
In particular
\[
HC_0((\Gami(\fA):I_{S(\fA)})=H_0(\Gamma/\cP:\elli(\fA)/[\elli(\fA):S(\fA)]).
\]
Specializing to $\fA=\C$ and using \eqref{intro:ce=estar} and \cite{kenetal2}*{Theorem 5.12} we obtain  
\begin{equation}\label{intro:agree}
HC_0(\Gami:I_{S})=S_\cE=HC_0(\cB:J_S)
\end{equation}
for every symmetric ideal $S\triqui\elli$. Another application of \eqref{intro:hcbh} is that for $\fA$ commutative the groups $HC_*(\Gami(\fA):I_{S(\fA)})$ carry a natural Hodge decomposition. Indeed, the usual Hodge decomposition of the cyclic chain complex \cite{lod} gives an $\emb$-equivariant direct sum decomposition
\[
HC((\elli(\fA):S(\fA))/\cP)=\bigoplus_{p\ge 0}HC^{(p)}((\elli(\fA):S(\fA))/\cP).
\]
Thus for 
\[
HC^{(p)}(\Gami(\fA):I_{S(\fA)})=\bH(\Gamma/\cP,HC^{(p)}((\elli(\fA):S(\fA))/\cP))
\]
we have  
\begin{equation}\label{intro:hodge}
HC_n(\Gami(\fA):I_{S(\fA)})=\bigoplus_{p=0}^nHC^{(p)}_n(\Gami(\fA):I_{S(\fA)}).
\end{equation}
In Theorem \ref{thm:hcrel} we obtain a description of $HC^{(p)}_n(\Gami:I_{S})$ in terms of differential forms which we shall presently explain. Let
$\Omega_{\elli}$ be the de Rham complex of absolute --i.e. $\Z$-linear-- algebraic differential forms. For $p\ge 0$ consider the subcomplex
\[
(\cF_p(S))^q=\left\{\begin{matrix}S^{p-q+1}\Omega^{q}_{\elli} & p\ge q\\
\Omega^q_{\elli}& q>p.\end{matrix}\right.
\]
We show in Theorem \ref{thm:hcrel} that 
\begin{equation}\label{intro:hcrel}
HC_*^{(p)}(\Gami:I_S)=\bH_{*+p}(\Gamma/\cP,\cF_{(p)}(S)).
\end{equation}
It follows that there is a spectral sequence (Corollary \ref{cor:wodspec})
\[
{}_pE^1_{m,n}=H_n(\Gamma/\cP,S^{m+1}\Omega^{p-m}_{\elli})\Rightarrow
HC_{m+n+p}^{(p)}(\Gami:I_S).
\]
Using this spectral sequence, we obtain (Corollary \ref{cor:milspec})
\[
HC_n^{(n)}(\Gami:I_S)=\left(S\Omega^n_{\elli}/d(S^2\Omega^{n-1}_{\elli})\right)_\cE
\]
for every symmetric ideal $S\triqui\elli$. In the particular cases \eqref{eq:lists} we can say more (see Proposition \ref{prop:compu}).
We show, for example, that if $p\in \Z$, then 
\begin{equation}\label{intro:compu}
HC_n^{(q)}(\Gami:I_{\ell^p})=\left\{\begin{matrix}0 & n<q+p-1\\
(\ell^1\Omega^{q-p}_{\elli}/d(\ell^{p/p+1}\Omega^{q-p}_{\elli}))_\cE & n=q+p-1.\end{matrix}\right.
\end{equation}
In particular, by \eqref{intro:min} and \eqref{intro:min2} we have 
\[
HC_{2p-2}(\cB:\Ell^p)=HC_{2p-2}(\Gami:I_{\ell^p})=HC_{2p-2}^{(p-1)}(\Gami:I_{\ell^p})=\ell^1_\cE.
\]

\goodbreak

The rest of this paper is organized as follows. In Section \ref{sec:prelis} we recall
some material from \cite{hl}, including, in particular, the crossed product decomposition
$I_{S(\fA)}=S(\fA)\#_\cP\Gamma$ (Proposition \ref{prop:cpg}). This crossed product is just 
the tensor product $S(\fA)\otimes_\cP\Gamma$ with multiplication
twisted by the action of $\emb$ on $S(\fA)$
\[
(a\# f)(b\# g)=af_*(b)\# fg.
\]  
In particular 
\[
\Gami(\fA)=I_{\elli(\fA)}=\elli(\fA)\#_\cP\Gamma).
\]
In Section \ref{sec:flat} we show that every two-sided ideal of $\Gami$ is flat (Proposition \ref{prop:nflatgamma}). 
Furthermore, if $S$ is closed under taking square roots of positive elements (e.g. if $S=c_0,\ell^{\infty-}$)
then $I_{S(\fA)}$ is a flat ideal of $\Gami(\fA)$ for every unital Banach algebra $\fA$ (Proposition \ref{prop:flatellia}). Section \ref{sec:flatp}
concerns the algebra $\cP$. We show that $\cP$ is a filtering colimit of separable $\Z$-algebras (Proposition \ref{prop:cpfil})
and that if $k$ is a field then $\cP(k)=\cP\otimes k$ is von Neumann regular (Corollary \ref{cor:cpfil}). Hence if $k$ is a field then
every $\cP(k)$-module is flat. Further, we show that for any unital ring $R$, $\Gamma(R)=\Gamma\otimes R$ is flat as a module over $\cP(R)$
(Proposition \ref{prop:gamaflatp}). The next section concerns excision. We call a ring $A$ $K$-excisive if it satisfies excision in algebraic
$K$-theory. It was proved by Suslin and Wodzicki (\cite{sw1}) that a ring having a certain triple factorization property (\textup{TFP}) is $K$-excisive.
We prove in Proposition \ref{prop:tfp} that if $\fA$ is a bornological algebra and $S\triqui\elli$ is a symmetric ideal such that $S(\fA)$ has the
\textup{TFP}, then $I_{S(\fA)}$ is $K$-excisive. This applies, for example, when $\fA$ is a $C^*$-algebra and $S=c_0$ (Example \ref{exa:excicstar}), and also when $\fA$ is a unital Banach algebra and $S=\ell^{\infty-}$ (Example \ref{exa:exciunital}). Section \ref{sec:homo} is concerned with the homology of crossed
products of the form $R\#_\cP\Gamma$ where $R$ is unital. The identity \eqref{intro:ce=estar} is proved in Proposition \ref{prop:ce=estar}. The quasi-isomorphism
\eqref{intro:quism} follows from the case $k=\Q$ of Example \ref{exa:hhp}, which says that if $k$ is a field, $A$ is a unital $\cP(k)$-algebra, and $N$ is an $A\otimes_{\cP(k)}A^{op}$-module, then the map of Hochschild complexes
\[
HH(A/k,N)\to HH(A/\cP(k),N)
\]
is a quasi-isomorphism. In Proposition \ref{prop:phiso} we compute the Hochschild homology of a crossed product $R\#_\cP\Gamma$ with coefficients
in a bimodule of the form $M\#_\cP\Gamma$. We show that there is a quasi-isomorphism
\[
\bH(\Gamma/\cP,HH(R/\cP(k),M))\weq HH(R\#_\cP\Gamma/\cP(k),M\#_\cP \Gamma).
\]
As an application, we obtain the isomorphism \eqref{intro:agree} in Corollary \ref{coro:hc0}. Using this, the calculations of \cite{wodk} compute $HC_0(\Gami:I_S)$ for $S\in\{\ell^p, \ell^{\pm p}\}$ (Lemma \ref{lem:previo}). Theorem \ref{thm:hccross} shows that if $k$ is a field and $R$ is unital
then there is a quasi-isomorphism
\[
\bH(\Gamma/\cP,HC(R/\cP(k)))\weq HC(R\#_\cP\Gamma/k).
\]
The identity \eqref{intro:hcbh} follows from this (Corollary \ref{coro:hccross}). In the particular case when $R$ is a commutative $\Q$-algebra, we obtain
(in Subsection \ref{subsec:hodge}) a
Hodge decomposition 
\[
HC_n(R\#_\cP\Gamma)=\bigoplus_{p=0}^nHC^{(p)}_n(R\#_\cP\Gamma)=\bigoplus_{p=0}^n\bH_n(\Gamma/\cP:HC^{(p)}(R/\cP)).
\]  
The decomposition \eqref{intro:hodge} follows from this. In Section \ref{sec:wod} we study the groups $HC_*(\Gami(\fA):I_{S(\fA)})$. 
The identities \eqref{intro:min} and \eqref{intro:min2} are proved in Proposition \ref{prop:hcmismo}. Theorem \ref{thm:k=kh} proves that the comparison map 
\eqref{intro:compa1} is an isomorphism when $\fA$ is a $C^*$-algebra and that \eqref{intro:compa2} is an isomorphism when $\fA$ is a unital Banach algebra.
The identity \eqref{intro:hcrel} is proved in Theorem \ref{thm:hcrel}. The latter is deduced from a computation of $HC_*^{(p)}(\elli/S)$ (Theorem \ref{thm:cgg}) which, we think, is of independent interest. The identity \eqref{intro:compu} is included in Proposition \ref{prop:compu}, which considers also the case when $p\notin\Z$ and computes some of the groups $HC_n^{(q)}(\Gami:I_{\ell^{\pm p}})$.

\begin{ack}
This article is part of an ongoing joint research project with Beatriz Abadie. It was originally part of our joint paper \cite{hl}, which we later decided to split into two articles, to facilitate publication. Although she had important contributions to the present article --particularly to Section \ref{sec:flat}-- she insisted in not being included as an author. I am indebted to her as well as to the Universidad de la Rep\'ublica for its hospitality during my many visits to Beatriz to collaborate
in this project over the last five years.    
\end{ack}

\section{Preliminaries}\label{sec:prelis}
\numberwithin{equation}{subsection}

\subsection{Symmetric sequence ideals and the algebra \topdf{$\Gami(\fA)$}{Gami(A)}}
Throughout this paper we work in the setting of bornological spaces
and bornological algebras; a quick introduction to the subject is
given in \cite{cmr}*{Chapter 2}. Recall that a (complete, convex)
bornological vector space over the field $\C$ of complex numbers is
a filtering union $\V=\cup_D\V_D$ of Banach spaces, indexed by the
disks of $\V$, such that the inclusions $\V_D\subset \V_{D'}$ are
bounded.
A subset of $\V$ is \emph{bounded} if it is a
bounded subset of some $\V_D$.
Let $X$ be a nonempty set. A map $X\to V$ is \emph{bounded} if its image is contained in a bounded
subset. We write
$\elli(X,\V)$  for the bornological vector
space of bounded maps $X\to \V$ where $B\subset\elli(X,\V)$ is bounded if
$\bigcup_{b\in B}b(X)$ is. 
The inverse monoid $\emb(X)$ of partially defined embeddings $X\to X$ acts on 
$\elli(X,\V)$ by means of the following action
\[
(f_*(\alpha)_x= \begin{cases}
                    \alpha_{f^{\dagger}(x)}&\text{ if $x\in \ran (f)$ }\\
                         0& \text{otherwise}   .      \end{cases}
\]
When $X=\N$ or $\V=\C$, we omit it from our notation; thus 
$\emb=\emb(\N)$, $\elli(\V)=\elli(\N,\V)$,  $\elli(X)=\elli(X,\C)$ and $\elli=\elli(\N,\C)$.  A subspace $S\triqui\elli$ is called \emph{symmetric} if it is stable under the action of $\emb$. If $S\subset \elli$ is a symmetric
subspace and $\V$ is a bornological vector space, then
\[
S(\V):=\{\alpha\in\elli(\V):(\exists D)\, \alpha(\N)\subset \V_D\text{
and }||\alpha||_D\in S\}
\]
is a symmetric subspace of $\elli(\V)$.

We will often work with sequences indexed by
infinite countable sets other than $\N$. A bijection $u:\N\to X$ gives rise
to a bounded isomorphism $\alpha\mapsto\alpha u$ between $\elli(X,\V)$ and $\ell^\infty(\V)$. If
$S\subset \ell^\infty$ is a symmetric subspace, we define
$S(X,\V)=\{s u^{-1}:s\in S(\V)\}$. Because $S$ is symmetric by assumption, this definition does
not depend on the choice of $u$. 

Recall a bornological algebra is
a bornological vector space $\fA$ with an associative bounded
multiplication. If $\fA$ is a bornological algebra, then pointwise
multiplication makes $\elli(\fA)$ into a bornological algebra, and if $S\triqui\elli$
is a symmetric ideal, then $S(\fA)\triqui\elli(\fA)$ is a symmetric two-sided ideal. 

Let $R$ be a ring and $A:\N\times\N\to R$ a countably infinite square matrix with entries in $R$. For $i,j\in\N$,
consider the following elements of $\Z\cup\{\infty\}$:
\begin{gather*}
r_i(A)=\#\{j: A_{ij}\neq 0\}, c_j(A)=\#\{i: A_{ij}\neq 0\},\\
N(A):=\sup\{r_i(A), c_i(A):i\in\N\}.
\end{gather*}
Let $\fA$ be a bornological algebra, and $S\triqui\elli(\fA)$ an ideal. Following \cite{hl}*{Definition 3.5}, we set
\begin{gather}\label{eq:mis}
I_{S(\fA)}=\{A=(\aij)_{i,j\in \NN}: \{\aij\}\in S(\N\times\N)\mbox{ and } N(A)<\infty\}\\ 
\text{ and }\Gami(\fA)=I_{\elli(\fA)}.\nonumber
\end{gather}

\subsection{Crossed products with \topdf{$\Gamma$}{Gamma}}

Let $R$ be a ring. \emph{Karoubi's cone} of the ring $R$ is the ring
\[
\Gamma(R)=\{A\in M_\N(R): N(A)<\infty \text{ and } \#\{A_{i,j}:(i,j)\in \N\times\N\}<\infty\}.
\]
We also consider the ring of all locally constant sequences
\[
\cP(R)=\{\alpha\in R^\N: \#\{\alpha_n:n\in\N\}<\infty\}.
\]
Observe that $\alpha\in\cP(R)$ if and only if the diagonal matrix $\diag(\alpha)\in\Gamma(R)$. We
shall identify $\cP(R)$ with $\diag(\cP(R))\subset\Gamma(R)$. 
When $R=\Z$ we omit it from our notation; we set
\[
\Gamma=\Gamma(\Z),\ \ \cP=\cP(\Z).
\]
By \cite{biva}*{Lemma 4.7.1} the map 
\begin{equation}\label{map:isotensog}
\phi:\Gamma\otimes R\to \Gamma(R),\ \ \phi(A\otimes x)_{i,j}=A_{i,j}x 
\end{equation}
 is an isomorphism. It
follows from this that $\Gamma$ and $\cP$ are flat $\Z$-modules. By 
\cite{hl}*{Remark 6.8} the restriction of $\phi$ induces an isomorphism 
\begin{equation}\label{map:isotensop}
\cP\otimes R\iso \cP(R).
\end{equation}
There is a monoid homomorphism
\begin{equation}\label{map:U}
U:\emb\to\Gamma,\ \ (U_f)_{i,j}=\left\{\begin{matrix} 1&\text{ if } j\in\dom(f) \text{ and }f(j)=i\\ 0 &\text{ otherwise. }\end{matrix}\right.
\end{equation}
Observe that the idempotent submonoid of $\emb$ is isomorphic to the monoid $2^\N$ of subsets of $\N$ with intersection of subsets as multiplication.
If $p^2=p$ and $A=\im p$, then $U_p=\diag(\chi_A)$ is a diagonal matrix. We will often identify $p$, $U_p$ and $\chi_A$. 
We also consider the monoid rings $\Z[2^\NN]$ and $\Z[\emb]$, and
the two-sided ideals
\begin{gather}
I=\langle\{\chi_{A\sqcup B}-\chi_A-\chi_B:A,B\subset\NN,\ \ A\cap B=\emptyset\}\rangle\triqui \Z[2^\NN],\label{eq:I}\\
J=\langle\{\chi_{A\sqcup B}-\chi_A-\chi_B:A,B\subset\NN,\ \ A\cap
B=\emptyset\}\rangle\triqui \Z[\emb].\label{eq:J}
\end{gather}

The following lemma follows from \cite{hl}*{Lemma 5.4 and Remark 6.8}.

\begin{lem}\label{lem:presfh}
Let $R$ be a ring. The maps \eqref{map:U}, \eqref{map:isotensog} and \eqref{map:isotensop} induce the following isomorphisms:
\item{i)} $\cP(R)=R[2^\NN]/R\otimes I$.
\item{ii)}$\fH(R)=R[\emb]/R\otimes J$.
\end{lem}

\begin{rem}\label{rem:tight}
Given a monoid $M$ and a unital ring $R$, a representation of $M$ in $R$-modules is
the same thing as a module over the monoid algebra $R[M]$. In view of
Lemma \ref{lem:presfh}, the modules over $\cP(R)$ and $\Gamma(R)$
correspond to those representations of the inverse monoids
$2^\N$ and $\emb$ which are tight in the sense of Exel
(see \cite{ruy}*{Def. 13.1 and Prop. 11.9}).
\end{rem}

Because $\emb$ is a monoid, if $\cA$ is a ring on which $\emb$
acts by algebra endomorphisms we can form the \emph{crossed product}
$\cA\#\emb$. As an abelian group, $\cA\#\emb=\cA\otimes_\Z\Z[\emb]$
with multiplication given by
\begin{equation}\label{eq:cpform}
(a\# f)(b\# g)=af_*(b)\# fg.
\end{equation}
Here $\#=\otimes$ and $f_*(b)$ denotes the action of $f$ on $\emb$.
Now assume that the $\emb$-ring $\cA$ is also a $\cP$-algebra,
that is, it is a ring and a $\cP$-bimodule, and these operations
are compatible in the sense that
\[
(ap)b=a(pb)\ \ (a,b\in \cA,\ \ p\in\cP).
\]
Further assume that $\cA$ is central as a $\cP$-bimodule,
i.e. $pa=ap$ ($a\in \cA$, $p\in\cP$), and that
\[
pa=p_*(a)\qquad (p\in 2^\N).
\]
Under all these conditions, we say that $\cA$ is an
\emph{$\emb$-bundle} (cf. \cite{busex}*{Def. 2.10}). For
$J\triqui\Z[\emb]$ as in \eqref{eq:J}, we have
\begin{gather*}
\cA\#\emb\troqui \cA\#J=\mspan\{r\# j:r\in \cA, j\in J\}\mbox{ and }\\
\cA\#\emb\troqui L=\mspan\{rp\#h-r\#ph:r\in \cA, p\in\cP, h\in
\emb\}.
\end{gather*}
Set
\begin{equation}\label{eq:crossdef}
\cA\#_\cP\Gamma=\cA\#\emb/(L+\cA\#J).
\end{equation}
Thus, $\cA\#_\cP\Gamma=\cA\otimes_\cP \Gamma$ as left
$\cP$-modules, and the product is that induced by \eqref{eq:cpform};
we have
\begin{equation}\label{eqprodcp}
(a\# U_f)(b\# U_g)=a f_*(b)\# U_{fg}\in \cA\#_\cP\Gamma.
\end{equation}

\begin{prop}\label{prop:cpg}(\cite{hl}*{Proposition 6.11})
Let $\fA$ be a bornological algebra. The map
\begin{equation}\label{map:cp}
\elli(\fA)\#_\cP\fH\to \Gami(\fA),\quad \alpha\# U_f\mapsto \diag(\alpha) U_f
\end{equation}
is an isomorphism of $\cP$-algebras. If $S\triqui\elli$ is a
symmetric ideal, then \eqref{map:cp} sends $S(\fA)\#_\cP\fH$
isomorphically onto $I_{S(\fA)}\triqui\Gami(\fA)$.
\end{prop}
\section{Flat ideals of \topdf{$\Gami$}{Gami} and \topdf{$\elli$}{elli}}\label{sec:flat}
\numberwithin{equation}{section}
\begin{prop}\label{prop:ellippal}
Every finitely generated ideal of $\elli$ is principal and
projective.
\end{prop}
\bpf The fact that the finitely generated ideals of $\elli$ are
projective follows from \cite{hmod}*{Corollary 2.4}. We will prove
that they are principal. Given $\alpha\in\elli$, set
\begin{equation}\label{eq:nualfa}
 \nu_{\alpha}(n)=
\begin{cases}
0, &\text{if $\alpha(n)=0$}\\
\frac{\alpha(n)}{|\alpha(n)|}, &\text{otherwise.}
\end{cases}
\end{equation}
Notice that
$\nu_{\alpha}$ is the partial isometry in the polar decomposition of
$\alpha$. In fact, we have
\[\alpha=  \nu_{\alpha}|\alpha|,\quad |\alpha|=  \overline{\nu}_{\alpha}\alpha.\]
It follows that, for any ideal $I$ in $\ell^{\infty}$, $\alpha\in I$ if and only if $|\alpha|\in I$.
Now let $I$ be an ideal of $\elli$ generated by $\{\alpha_0,\alpha_1\}$,
and set
\[\mu(n)=\max\{|\alpha_0(n)|,|\alpha_1(n)|\}. \]
For $i=0,1$, let
\[
\gamma_i(n)=\left\{\begin{matrix} 1/2 & \text{ if } |\alpha_0(n)|=|\alpha_1(n)|\\
                                   1  & \text{ if } |\alpha_i(n)|>|\alpha_{1-i}(n)|\\
                                   0 & \text{ otherwise.}\end{matrix}\right.
\]
We have $\mu=\gamma_0|\alpha_0|+\gamma_1|\alpha_1|$;
thus $\mu\in I$. Now set
\[\tau_i(n)=
\begin{cases}
0 & \text{ if } \mu(n)=0\\
\frac{\alpha_i(n)}{\mu(n)} &\text{ otherwise.}
\end{cases}\]
Then $\alpha_i=\tau_i\mu$, $(i=0,1)$. Notice that $\tau_i\in
\elli$, since $|\tau_i(n)|\leq 1$ for all $n\in\NN, i=0,1$.
Therefore, $\mu$ generates $I$. The general case can now be proven by
induction on the number of generators. \epf

\begin{coro}\label{coro:flatelli}
Every ideal of $\elli$ is flat.
\end{coro}

\begin{prop}\label{prop:flatsa}
Let $\fA$ be a unital Banach algebra and $S\triqui\elli$ a symmetric
ideal. Assume that
\[
\alpha\in S\Rightarrow \sqrt{|\alpha|}\in S.
\]
Then $S(\fA)\triqui\elli(\fA)$ is flat both as a
right and as a left $\elli(\fA)$-module.
\end{prop}
\begin{proof}
Consider the following homomorphism of $\elli(\fA)$-modules
\[
\mu:\elli(\fA)\otimes_{\elli} S\to S(\fA),\ \
\mu(\alpha\otimes\beta)_n=\alpha_n\beta_n.
\]
We claim that $\mu$ is an isomorphism. To prove it is surjective,
for $\alpha\in S(\fA)$ let $\nu_{\alpha}$ be as in \eqref{eq:nualfa}.
Then $\nu_\alpha\in\elli(\fA)$ and
\[
\alpha=\mu(\nu_\alpha\otimes||\alpha||).\
\]
Thus $\mu$ is surjective. To prove it is also injective, let
\[
\eta=\sum_{i=1}^n\alpha^i\otimes \beta^i\in\ker\mu.
\]
By Proposition \ref{prop:ellippal}, the ideal
$\langle\beta^1,\dots,\beta^n\rangle\triqui\elli$ is principal. Let
$\beta$ be a generator; we may and do choose it so that
$\beta=|\beta|$. By bilinearity, we may rewrite $\eta$ as a single
elementary tensor and we have
\[
\eta=\alpha\otimes\beta,\ \ \alpha\beta=0.
\]
But $\alpha\beta=0$ implies $\alpha\sqrt{\beta}=0$, whence
\[
\eta=\alpha\sqrt{\beta}\otimes\sqrt{\beta}=0.
\]
Thus the claim is proved. It follows that $S(\fA)$ is flat as a left
$\elli(\fA)$-module, since it is the scalar extension of $S$, which
is a flat $\elli$-module by Corollary \ref{coro:flatelli}. The proof
that $S(\fA)$ is flat on the right is similar.
\end{proof}

\begin{exas}\label{exas:root}
The hypothesis of Proposition \ref{prop:flatsa} are satisfied, for
example, when $S$ is either of $\ell^{\infty-}$, $c_0$.
\end{exas}

\begin{prop}\label{prop:nflatgamma}
Every two-sided ideal of $\Gami$ is flat both as a left and as a
right $\Gami$-module.
\end{prop}
\bpf Let $I\triqui\Gami$. By \cite{hl}*{Theorem 4.5} there is a
symmetric ideal $S$ such that $I=I_S$. Observe that
\[
I_S=S\otimes_\cP\Gamma=S\otimes_{\elli}\elli\otimes_{\cP}\Gamma=S\otimes_{\elli}\Gami.
\]
Thus $I_S\otimes_{\Gami}=S\otimes_{\elli}$ is exact by Corollary
\ref{coro:flatelli}. Hence $I$ is flat as a right module and
therefore also as a left module, since $\Gami$ is a $*$-algebra.
\epf
\begin{rem}\label{rem:minfiproye}
By \cite{hl}*{Proposition 4.6}, if $k$ is a field, then
$M_\infty k$ is the only proper two-sided ideal of $\Gamma(k)$.
Observe that $M_\infty k$ is projective both as a left and as a
right module, since it is isomorphic to an infinite sum of copies of
the principal ideal generated by the idempotent $E_{1,1}$.
\end{rem}

\begin{prop}\label{prop:flatellia}
Let $\fA$ be a unital Banach algebra and $S\triqui\elli$ a symmetric
ideal as in Proposition \ref{prop:flatsa}. Then $I_{S(\fA)}$ is flat
both as a left and as a right $\Gami(\fA)$-module.
\end{prop}
\begin{proof}
By Proposition \ref{prop:cpg} and the proof of Proposition
\ref{prop:flatsa} we have the following canonical isomorphisms of
right $\Gami(\fA)$-modules
\[
I_{S(\fA)}=S(\fA)\otimes_{\cP}\Gamma= S\otimes_{\elli}
\elli(\fA)\otimes_{\cP}\Gamma=S\otimes_{\elli}\Gami(\fA).
\]
This, together with Corollary \ref{coro:flatelli}, proves that
$I_{S(\fA)}$ is flat as a right $\Gami(\fA)$-module. The proof that
it is also flat on the left is similar.
\end{proof}

\section{Flatness properties of \topdf{$\cP$}{cP}}\label{sec:flatp}

Let $k$ be a commutative ring. Recall that a $k$-algebra $A$ which is projective as an $A\otimes_k
A^{op}$-module is called \emph{separable}.
\begin{prop}\label{prop:cpfil}
The $k$-algebra $\cP(k)$ is a filtering union of separable algebras.
\end{prop}

\begin{proof} We shall show that $\cP$ is a filtering union of finite products of copies of $\Z$, indexed by
the
finite partitions of $\N$. Here a finite partition of $\N$ is a
finite set $\pi=\{A_1,\dots,A_n\}$ of subsets of $\N$ such that
$\N=A_1\sqcup\dots\sqcup A_n$. We say that a partition
$\rho=\{B_1,\dots,B_m\}$ is \emph{finer} than $\pi$ if the following
condition is satisfied:
\[
(\forall 1\le i\le m)(\exists j)\quad B_i\subset A_j.
\]
Note that if $\pi$ and $\pi'$ are any two finite partitions, then
\[
\pi\land \pi'=\{B\subset\N:(\exists A\in\pi,A'\in\pi')B=A\cap A'\}.
\]
is a finite partition and is finer than each of them. Thus
the set
\[
\mathrm{Part}(\N)=\{\pi \mbox{ finite partition of $\N$ }\}.
\]
is a filtered partially ordered set. If $\pi\in \mathrm{Part}(\N)$ has $n$ elements,
put
\[
\cP\supset R_\pi=\bigoplus_{i=1}^n \Z P_{A_i}.
\]
Observe that $R_\pi\cong \Z^n$ and that $\cP=\bigcup_\pi R_\pi$. This proves the proposition
in the case $k=\Z$. The general case follows from this using the isomorphism $\cP\otimes k\iso \cP(k)$.
\end{proof}

\begin{coro}\label{cor:cpfil}
If $k$ is a field, then $\cP(k)$ is a von Neumann regular ring. In other words, every
$\cP(k)$-module is flat.
\end{coro}

\begin{prop}\label{prop:gamaflatp}
Let $R$ be a unital ring. Then $\Gamma(R)$ is flat, both as a left and as a right $\cP(R)$-module. 
\end{prop}
\begin{proof}
We prove that $\Gamma(R)$ is flat as a right $\cP(R)$-module; the proof that it is also flat on the left is similar.
If $M$ is a $\cP(R)$-module, then 
\[
\Gamma(R)\otimes_{\cP(R)}M=\Gamma\otimes R\otimes_{\cP\otimes R}M=\Gamma\otimes_\cP M.
\]
Hence it suffices to consider the case $R=\Z$. In view of Proposition \ref{prop:cpfil} and its proof, we have
\[
\Gamma\otimes_\cP M=\colim_{\pi\in \mathrm{Part}(\N)}\Gamma\otimes_{R_\pi}M.
\]
Hence it suffices to show that $\Gamma$ is flat as a module over $R_\pi$, for each $\pi\in\mathrm{Part}(\N)$. We have
\[
R_\pi=\bigoplus_{A\in\pi} \Z P_{A}.
\]
Hence
 \[
\Gamma\otimes_{R_\pi}M=\bigoplus_{A\in\pi}\Gamma p_A\otimes p_AM.
\]
Thus it suffices to show that $\Gamma p_A$ is flat as an abelian group. Since $\Gamma p_A$ is a direct summand of $\Gamma$,
we are reduced to showing that $\Gamma$ is $\Z$-flat. As said above, the map \eqref{map:isotensog} is an isomorphism for every ring; in particular this
applies to show that if $M$ is any abelian group --regarded as a ring with trivial multiplication-- then $\Gamma\otimes M=\Gamma(M)$. Since $M\to \Gamma(M)$
is clearly exact, this conlcudes the proof.
\end{proof}

\section{Excision}\label{sec:exci}
A ring $A$ is called \emph{$K$-excisive} if for every ideal embedding $A\triqui B$ the map $K_*(A)\to K_*(B:A)$ is an isomorphism. It was proved by Suslin and Wodzicki (\cite{sw1}*{Theorem C}) that if a ring $A$ satisfies the following property then it is $K$-excisive.
\begin{gather*}
\forall n, \forall a\in A^{\oplus n}, \exists b\in
A^{\oplus n}, \ \ c,d \in A,\text{ such that } a=cdb \text{ and such that}\\
 (0:_Ad)_r:=\{v\in A:dv=0\}= (0:_Acd)_r.
\end{gather*}
The right ideal $(0:_Ad)_r$ is called the {\it right annihilator} of $d$ in $A$. The property above is the so-called left \emph{triple factorization property} (\textup{TFP}). A ring is $K$-excisive if and only if its opposite ring $A^{op}$ is (\cite{sw1}*{Remark (1) pp 53}), so rings satisfying the right \textup{TFP} are excisive also. Further results of Wodzicki (\cite{wodex}*{Theorems 1.1 and 3.1}) and
of Suslin-Wodzicki (\cite{sw1}*{Theorem B}) establish that a $\Q$-algebra $A$ is $K$-excisive if and only if it is excisive for cyclic homology, and that this happens if and only if the \emph{bar complex} $(C^{bar}_*(A),b')$ is exact. Here
\begin{gather*}
b':C^{bar}_{n+1}(A)=A^{\otimes n+2}\to A^{\otimes n+1}=C^{bar}_{n}(A)\qquad (n\ge 0)\\
b'(a_0\otimes\cdots\otimes a_{n+1})=
\sum_{i=0}^n(-1)^ia_0\otimes\cdots\otimes
a_ia_{i+1}\otimes\cdots\otimes a_{n+1}.
\end{gather*}
The tensor products above are taken over $\Z$ or, equivalently, over
$\Q$, since $A$ is assumed to be a $\Q$-algebra. The $\Q$-algebras
whose bar homology vanishes --that is, the $K$-excisive ones-- are
also called \emph{$H$-unital}.

\begin{prop}\label{prop:tfp}
Let $\fA$ be a bornological algebra and $S\triqui\elli$ a symmetric
ideal. Assume that $S(\fA)$ has the (left or right) triple
factorization property. Then $I_{S(\fA)}$ is $K$-excisive. 
\end{prop}
\begin{proof}
Assume that $S(\fA)$ has the left \textup{TFP}. We have to prove that $I_{S(\fA)}$ is $H$-unital.
Let $n\ge 0$ and let $z\in C_n^{bar}(I_{S(\fA)})$ be a cycle. We may
write
\[
z=\sum_{i=1}^m\diag(\alpha^{0,i})U_{f_{0,i}}\otimes\cdots\otimes
\diag(\alpha^{n,i})U_{f_{n,i}},
\]
where $\supp(\alpha^{j,i})=\ran(f_{j,i})$ for all $i,j$. By
\textup{TFP}, there are elements $\gamma$, $\delta$ and
$\beta^1,\dots,\beta^m$ in $S(\fA)$ such that
$\alpha^{0,i}=\gamma\delta\beta^i$ $(1\le i\le m)$, and such that
\begin{equation}\label{eq:annil1}
(0:_{S(\fA)}\gamma\delta)_r=(0:_{S(\fA)}\delta)_r.
\end{equation}
Now observe that if $\theta\in S(\fA)$ then, by our definition of $I_{S(\fA)}$ \eqref{eq:mis}, we
have
\[
(0:_{I_{S(\fA)}}\diag(\theta))_r=\{T\in I_{S(\fA)}: (\forall j)\ \
T_{*,j}\in (0:_{S(\fA)}\theta)_r\}.
\]
Hence, \eqref{eq:annil1} implies that
\begin{equation}\label{eq:annil2}
(0:_{I_{S(\fA)}}\diag(\gamma\delta))_r=(0:_{I_{S(\fA)}}\diag(\delta))_r.
\end{equation}
Put
\[
y=\sum_i\diag(\beta^{i})U_{f_{0,i}}\otimes\diag(\alpha^{1,i})U_{f_{1,i}}\otimes\cdots
\otimes \diag(\alpha^{n,i})U_{f_{n,i}}.
\]
Consider the following element of $C^{bar}_{n+1}(I_{S(\fA)})$
\[
w=\diag(\gamma)\otimes \diag(\delta)y.
\]
We have
\[
b'(w)=z-\diag(\gamma)\otimes \diag(\delta)b'(y).
\]
If $n=0$ then $b'(y)=0$, so this proves that $z$ is a boundary. We
have to show that $\diag(\delta)b'(y)=0$ if $n\ge 1$. Choose a basis
$\{v_l\}$ of the $\Q$-vector space $C_{n-1}^{bar}(I_{S(\fA)})$. Then
$y=\sum_lT_l\otimes v_l$ for unique $T_l\in I_{S(\fA)}$, and
\[
0=b'(z)=\diag(\gamma\delta)
b'(y)=\sum_l\diag(\gamma\delta)T_l\otimes v_l.
\]
Hence we must have $\diag(\gamma\delta)T_l=0$ for all $l$, and therefore $\diag(\delta)b'(y)=0$ by \eqref{eq:annil2}.
\end{proof}

\begin{exa}\label{exa:excicstar}
Any Banach algebra with a bounded left approximate unit satisfies
the Cohen-Hewitt factorization property; thus it has the left
\textup{TFP} (\cite{friendly}*{Lemma 6.5.1}). In particular, this
applies to $C^*$-algebras. If $\fA$ is a $C^*$-algebra then $c_0(\fA)$ is
again a $C^*$-algebra; hence $I_{c_0(\fA)}$ is $K$-excisive,
by Proposition \ref{prop:tfp}.
\end{exa}

\begin{exa}\label{exa:exciunital}
If $\fA$ is a unital Banach algebra then $\ell^{\infty-}(\fA)$ has
the \textup{TFP}. To see this, let
$\alpha^1,\dots,\alpha^m\in\ell^{\infty-}$. Choose $p$ such that
$\alpha^i\in\ell^p(\fA)$ for all $i$. For each $n$ put
\[
\gamma_n=\max_{1\le i\le m}{||\alpha^i_n||},\ \
\beta^i_n=\left\{\begin{matrix}\alpha^i_n/\gamma_n^{1/2}& \text{ if
} \gamma_n\ne 0\\ 0 &\text{otherwise.}\end{matrix}\right.
\]
Then $||\beta_n^i||\le ||\alpha_n^i||^{1/2}$ and therefore $\beta^i\in\ell^{2p}(\fA)$. Similarly $\gamma^{1/4}\in\ell^{4p}(\fA)$. One checks that the factorization $\alpha^i=\gamma^{1/4}\gamma^{1/4}\beta^i$ satisfies
the requirements of the \textup{TFP}.
\end{exa}

\section{Homology of crossed products with \topdf{$\Gamma$}{gama}}\label{sec:homo}

\numberwithin{equation}{subsection}
\subsection{Homology of augmented algebras}\label{subsec:suple}

In this subsection $A$ and $B$ will be unital rings; furthermore, $B$ will be an
$A$-algebra, that is, $B$ will be a ring together with a
unital ring homomorphism $\iota:A\to B$. 
Further assume that $A$ is equipped with a left $B$-module structure and a
surjective $B$-module homomorphism $\pi:B\to A$ such that
$\pi\iota=id_A$. Observe that the triple $(B,A,\pi)$ is an augmented
ring in the sense of Cartan-Eilenberg \cite{carteil}*{Chapter
VIII,\S1}. Since in addition, $B$ is an $A$-algebra, we call the
triple $(B,A,\pi)$ an \emph{augmented algebra}. Let $M$ be a right
$B$-module. Consider the simplicial $A$-module $\perp(B/A, M) $
given in dimension $n$ by
\[
\perp_n(B/A, M)=M\otimes_A B^{\otimes_A n},
\]
with face and degeneracy maps defined as follows $(n\ge 0)$
\begin{gather*}
\partial_i:\perp_{n+1}(B/A, M)\to \perp_{n}(B/A, M),\\
\partial_i(x_0\otimes\dots\otimes x_{n+1})=\left\{\begin{matrix}x_0\otimes\dots\otimes x_ix_{i+1}\otimes\dots \otimes x_{n+1}& i\le n\\ x_0\otimes\dots\otimes x_{n}\pi(x_{n+1})& i=n+1\end{matrix}\right.\\
\delta_i:\perp_{n}(B/A, M)\to \perp_{n+1}(B/A, M),\ \ (0\le i\le n)\\
\delta_i(x_0\otimes\dots\otimes x_n)= x_0\otimes\dots\otimes
x_i\otimes 1\otimes x_{i+1}\otimes\dots\otimes x_n.
\end{gather*}
The homology of $(B/A, M)$ relative to $(A,B,\pi)$, denoted $H_*(B/A,M)$, is the homotopy of the simplicial module
$\perp (B/A, M)$;
\[
H_*(B/A,M)=\pi_*(\perp (B/A, M))=H_*(\perp (B/A, M),\partial).
\]
Here
\[
\partial=\sum_{i=0}^{n+1}(-1)^i\partial_i:\perp_{n+1} (B/A, M)\to \perp_n (B/A, M)
\]
is the alternating sum of the face maps. We have
\[
H_0(B/A,M)=M\otimes_BA.
\]
Let $P(B/A)=\perp (B/A,B)$; $\pi:P(B/A)\to A$ is a resolution which is projective relative to $B/A$, and $\perp(B/A,N)=N\otimes_BP(B/A)$.
Hence if $B$ is flat both as a left and as a right $A$-module, then
\[
H_*(B/A,M)=\tor_*^B(M,A).
\]
Without flatness assumptions, we may regard the groups $H_*(B/A,M)$ as relative $\tor$ groups. 
\begin{lem}\label{lem:moritesuple}
Let $N$ be a right $B$-module. Consider $N^2=N^{1\times 2}$ as a right module over
$M_2B$ via the matrix product. View $M_2B$ as an $A\oplus A$-algebra through
the diagonal embedding $(a_1,a_2)\mapsto E_{11}a_1+E_{22}a_2$. Then the map
 \begin{gather*}
\iota:\perp(B/A,N)\to \perp (M_2(B)/A\oplus A, N\oplus N)\\
\iota(x_0\otimes\dots\otimes x_n)=E_{11}x_0\otimes\dots\otimes E_{11}x_n
\end{gather*}
is a quasi-isomorphism.
\end{lem}
\begin{proof}
Consider the maps
\begin{gather*}
\iota':P(B/A)^{2\times 1}\to P(M_2B/A^2),\\
\iota'(E_{i1}( x_0\otimes\dots\otimes x_n))=E_{i1}x_0\otimes E_{11}x_1\otimes\dots\otimes E_{11}x_n,\\
\mbox{ and }p':P(M_2B/A^2)\to P(B/A)^{2\times 1},\\
p'(E_{i_0,i_1}x_0\otimes\dots\otimes
E_{i_n,i_{n+1}}x_n)=E_{i_01}(x_0\otimes\dots\otimes x_n).
\end{gather*}
One checks that both $\iota'$ and $p'$ are $M_2B$-linear chain homomorphisms, and that $p'\iota'=1$. In particular
$\pi^{2\times 1}:P(B/A)^{2\times 1}\to A^{2\times 1}$ is a projective resolution relative to $M_2A/A^2$, whence
$$\iota=N^{1\times 2}\otimes_{M_2B}\iota'$$ is a quasi-isomorphism, as claimed.
\end{proof}

\subsection{The augmented algebra \topdf{$(\Gamma,\cP,\epsilon_l)$}{Gamma,cP,epsilonl}}\label{subsec:gammasuple}

Regarding the elements of $2^\N$ as sequences of zeros and ones,
there is an obvious action $\emb\times 2^\N\to 2^\N$, $(f,p)\mapsto f_*(p)$. It
agrees with the inner action; we have 
\[
f_*(p)=fpf^\dagger.
\] 
Thus $\Z[2^\N]$ is a $\Z[\emb]$-module. Note that, if $A,B\subset\N$ are disjoint, then
for $I\subset \Z[2^\N]$ as in \eqref{eq:I} and $q\in 2^\N$, we have
\begin{multline*}
f_*((p_{A\sqcup B}-p_A-p_B)q)=\\
\left(p_{f((A\sqcup
B)\cap\dom(f))}-p_{f(A\cap\dom(f))}-p_{f(B\cap\dom(f))}\right)f_*(q)\in
I,
\end{multline*}
\begin{align*}
(f (p_{A\sqcup B}-p_A-p_B)g)_*(q)&=f_*((p_{A\sqcup B}-p_A-p_B)_*(g_*(q)))\\
&=f_*((p_{A\sqcup B}-p_A-p_B)\cdot g_*(q))\in I.
\end{align*}
Thus $\cP$ is a $\Gamma$-module. Let $f\in\emb$; put
\[
\epsilon_l(f)=p_{\ran(f)}\in 2^\N\owns
\epsilon_r(f)=\epsilon_l(f^\dagger)=p_{\dom(f)}.
\]
Note that
\[
\epsilon_l(fg)(n)=p_{\ran(fg)}(n)=\left\{\begin{matrix} 1& \mbox{ if } n\in f(\dom(f)\cap\ran(g))\\
0&\mbox{ otherwise}\end{matrix}\right\}=f_*(\epsilon_l(g))(n).
\]
Thus the induced linear map $\epsilon_l:\Z[\emb]\to \Z[2^\N]$ is a
homomorphism of left $\Z[\emb]$-modules. In particular, if
$A,B\subset\N$ are disjoint, we have
\[
\epsilon_l(f(p_{A\sqcup B}-p_A-p_B)g)=f_*(p_{A\sqcup
B}-p_A-p_B)\epsilon_l(g)\in I.
\]
Hence $\epsilon_l$ induces a homomorphism of left $\Gamma$-modules
\[
\epsilon_l:\Gamma\to \cP.
\]
Observe that the canonical inclusion $\cP\subset\Gamma$, which is an
algebra homomorphism, but not a $\Gamma$-module homomorphism, is a
section of $\epsilon_l$. Thus we are in the augmented algebra setting described
above. Moreover $\Gamma$ is flat over $\cP$, by Proposition \ref{prop:gamaflatp}. Hence
\begin{equation}\label{eq:hgammator}
H_*(\Gamma/\cP,M)=\tor_*^{\Gamma}(M,\cP).
\end{equation}
Observe also that if $k$ is any commutative ring and $M$ is a $\Gamma(k)$-module, then
\[C(\Gamma/\cP,M)=C(\Gamma(k)/\cP(k),M).\]
In particular, 
\[
H_*(\Gamma/\cP,M)=H_*(\Gamma(k)/\cP(k),M). 
\]
In the next lemma and below we consider the following submonoids of
$\emb$
\[
\emb\supset\cE=\{f:\dom f=\N\}\supset\estar=\{f\in\cE:\ran(f)=\N\}.
\]
If $M$ is a $\Gamma$-module and $\fS\in\{\cE,\estar\}$ we write
\[
M_\fS=M/\mspan\{m-f_*(m):f\in\fS\}.
\]
Here the $\mspan$ is $\Z$-linear.

\begin{lem}\label{lem:kerepsi}
The kernel of $\epsilon_l:\Gamma\to \cP$ is generated, as a left
$\cP$-module, by the elements $U_{f}- 1$, $f\in\estar$.
\end{lem}
\begin{proof}
Let $K=\ker(\epsilon_l)$. It is clear that $K$ is generated, as an
abelian group, by the elements $U_f-p_{\ran f}$, $f\in\emb$. Assume
that $f\in\emb$ but $f\notin \estar$. We claim that we may choose a
subset $A\subset \dom(f)$ such that $B=\N\backslash A$ is bijectable
to $\N\backslash f(A)$, and such that $\N\backslash (\dom f\cap B)$
is bijectable to $\N\backslash f(\dom f\cap B)$. Indeed if
$\N\backslash\dom f$ is already bijectable to $\N\backslash\ran f$,
we may take $A=\dom f$. Otherwise $\dom f$ is infinite, so we may
split it into two disjoint infinite pieces, and take $A$ to be one
of them. Thus the claim is proved. For such $A$, there exist
$g,h\in\estar$ such that $g_{|A}=f_{|A}$ and $h_{|\dom(f)\cap
B}=f_{|\dom(f)\cap B}$. We have
\begin{gather*}
p_{\ran f}=p_{f(A)}+p_{f(\dom f\cap B)}\mbox{ and}\\
 U_f=p_{f(A)}U_{f_{|A}}+p_{f(\dom(f)\cap
B)}U_{f_{\dom(f)\cap B}}=p_{f(A)}U_g+p_{f(\dom(f)\cap B)}U_h.
\end{gather*}
Thus
\[
U_f-p_{\ran f}=p_{f(A)}(U_g-1)+p_{f(\dom f\cap B)}(U_h-1).
\]
\end{proof}

\begin{prop}\label{prop:ce=estar}
Let $M$ be a $\Gamma$-module.  Then
$$H_0(\Gamma/\cP,M)=M_\cE=M_{\estar}.$$
\end{prop}
\begin{proof} Immediate from Lemma \ref{lem:kerepsi}.
\end{proof}

\subsection{Hochschild homology}\label{subsec:HH}
We recall the basic definitions for Hochschild homology of algebras
over a noncommutative base ring (\cite{lod}*{\S1.2.11}). If $N$ is
a $B\otimes B^{op}$-module, we write
\begin{align*}
[b,x]=& bx-xb \qquad (b\in B,x\in N),\\
[B,N]=& \{\sum_{i=1}^n [b_i,x_i]:b_i\in B, x_i\in N, n\ge 1\},\\
N_B=& N/[B,N].
\end{align*}
Next let $A\to B$ be a unital ring homomorphism. Recall from
\cite{lod}*{\S 1.2.11} that the {\em Hochschild} homology of $B$
relative to $A$ with coefficients in $N$, $HH_*(B/A,N)=\pi_*C(B/A,M)$, is
the homotopy of the simplicial $\Z$-module which is
given in dimension $n$ by
\[
C_n(B/A,N)=(N\otimes_A B^{\otimes_A n})_A,
\]
with the following face and degeneracy maps
\begin{gather*}
\mu_i:C_{n+1}(B/A, N)\to C_{n}(B/A, N),\\
\mu_i(x_0\otimes\dots\otimes x_{n+1})=\left\{\begin{matrix}x_0\otimes\dots\otimes x_ix_{i+1}\otimes\dots\otimes x_{n+1}& i\le n\\ x_{n+1}x_0\otimes\dots\otimes x_{n}& i=n+1\end{matrix}\right.\\
s_i:C_n(B/A,N)\to C_{n+1}(B/A, N),\ \ (0\le i\le n)\\
s_i(x_0\otimes\dots\otimes x_n)=x_0\otimes\dots\otimes x_i\otimes
1\otimes x_{i+1}\otimes\dots\otimes x_n.
\end{gather*}
We write $b$ for the alternating sum of the face maps, and
$HH(B/A,N)$ for the resulting chain complex. Thus
$$HH_*(B/A,N)=H_*(HH(B/A,N))$$ is the Hochschild homology of $B/A$
with coefficients $N$. If $A$ is commutative and $B$ is central as
an $A$-bimodule, then $B\otimes_A B^{op}$ is a ring. If furthermore,
$B$ happens to be flat as a left $A$-module, then
\[
HH_*(B/A,N)=\tor^{B\otimes_AB^{op}}_*(B, N).
\]
Note this is the case, for example, if $A$ is a field. We shall write
$HH_*(B,N)$ for $HH_*(B/\Z,N)$.

\begin{rem}\label{rem:hhm} If $A$ and $B$ are commutative and $M$ is a central bimodule, then $C(B/A,M)=M\otimes_BC(B/A,B)$.
\end{rem}
\begin{lem}\label{prop:hhp} (cf. \cite{lod}*{Theorem 1.12.13})
Let $k$ be a field, $A\to B$ a homomorphism of unital $k$-algebras, and $N$ a $B\otimes_kB^{op}$-module. Assume that
$A$ is a filtering colimit of separable $k$-algebras. Then
\[
HH_*(B/k,N)=HH_*(B/A,N).
\]
\end{lem}
\begin{proof} It suffices to show that $B\otimes_{A}B^{op}$ is flat as a $B\otimes_k B^{op}$-module.
By hypothesis $A=\colim_i A_i$ is a filtering colimit of separable
algebras. Hence $B\otimes_{A}B^{op}=\colim_i B\otimes_{A_i} B^{op}$,
so it suffices to prove that if $k\subset A$ is separable then
$B\otimes_A B$ is flat over $B\otimes_k B^{op}$, and this is
well-known.
\end{proof}

\begin{exa}\label{exa:hhp}
If $k$ is a field, $A$ is a unital $\cP(k)$-algebra, and $N$ is an $A\otimes_k A^{op}$-module, then
$HH_*(A/k,N)=HH_*(A/\cP(k),N)$, by Proposition \ref{prop:cpfil} and Lemma
\ref{prop:hhp}. If $A\supset\Q$, then $HH_*(A,N)=HH_*(A/\Q,N)$ and $HH_*(A/\cP,N)=HH_*(A/\cP(\Q),N)$,
whence we also have $HH_*(A,N)=HH_*(A/\cP,N)$.
\end{exa}

\subsection{Hochschild homology of crossed products with \topdf{$\Gamma$}{Gamma}}\label{subsec:HRbundle}
In this subsection $k$ is a field and, as in \eqref{eq:crossdef}, $R$ is an \emph{$\emb$-bundle over $k$}; that is, $R$ is a $k$-algebra with
a $k$-linear action of $\emb$ so that $R$ is an $\emb$-bundle.
We also fix an $R$-bimodule $M$, central as a
$\cP$-bimodule, together with a left action of $\emb$
\[
\emb\times M\to M, \ \ (f,m)\mapsto f_*(m).
\]
We require that this action induce a $\Gamma$-module structure on
$M$ which is \emph{covariant} in the sense that
\begin{equation}\label{eq:covariant}
f_*(rms)=f_*(r)f_*(m)f_*(s)\quad (r,s\in R, m\in M).
\end{equation}
In this
situation, we can form the crossed product $M\#_{\cP}\Gamma$; this
is the $R\#_\cP\Gamma$-bimodule consisting of  $M\otimes_\cP\Gamma$
equipped with the following left and right actions of
$R\#_\cP\Gamma$
\[
(a\# U_f)(m\# U_g)=af_*(m)\#U_{fg},\qquad (m\# U_g)(a\#
U_f)=mg_*(a)\# U_{gf}.
\]
Observe that, as $R$ is assumed to be a $k$-algebra, $M\#_\cP\fH=M\#_{\cP(k)}\fH(k)$. 
We are interested in the Hochschild homology of $R\#_\cP\Gamma$ with
coefficients in $M\#_{\cP}\Gamma$, which by Example \ref{exa:hhp} is
computed by the simplicial $\cP(k)$-module $C(R\#_\cP\Gamma/\cP(k),M\#_\cP
\Gamma)$. On the other hand it is not hard to check, using
\eqref{eq:covariant} and the definition of $\emb$-bundle, that the
diagonal action of $\emb$ on $C(R/k)$ descends to an action of
$\Gamma$ on $C(R/\cP(k))$. Hence we may also consider the bisimplicial
module $\perp(\Gamma/\cP, C(R/\cP(k),M))$ which results from applying
the functor $\perp(\Gamma/\cP,-)$ dimension-wise to the simplicial
module $C(R/\cP(k),M)$. The diagonal of this bisimplicial module is
\begin{multline*}
\diag(\perp(\Gamma/\cP, C(R/\cP(k),M)))_n=\\
\perp^n(\Gamma/\cP, C_n(R/\cP(k),M))=\left(M\otimes_{\cP} R^{\otimes_{\cP(k)}
n}\right)_\cP\otimes_\cP \Gamma^{\otimes_\cP n},
\end{multline*}
with faces $\mu_i\partial_i$ and degeneracies $s_i\delta_i$. The
simplicial module \[\diag(\perp(\Gamma/\cP, C(R/\cP(k),M)))\] is a
model for the hyperhomology of $\Gamma/\cP$ with $C(R/\cP(k),M)$ coefficients.
Hence, if $\bH(\Gamma/\cP,C(R/\cP(k),M))$ is any other
such model, we have a quasi-isomorphism
\[
\bH(\Gamma/\cP,C(R/\cP(k),M))\weq\diag(\perp(\Gamma/\cP, C(R/\cP(k),M)).
\]
Observe that any element of $\diag(\perp(\Gamma/\cP, C(R/\cP(k),M)))_n$
can be written as a sum of congruence classes of elementary tensors
of the form
\begin{equation}\label{eq:eltens}
x=a_0\otimes a_1\otimes\dots\otimes a_n\otimes f_1\otimes
\dots\otimes f_{n},
\end{equation}
where $a_0\in M$, $a_i\in R$, and  $f_i\in\emb$ $(i\ge 1)$ are such that
\begin{gather*}
\epsilon_r(f_i)=\epsilon_l(f_{i+1})\quad (1\le i\le n-1),\\
a_j\epsilon_l(f_1)=a_j \quad(0\le j\le n).
\end{gather*}
Next we define a map
\[
\phi:\diag(\perp(\Gamma/\cP, C(R/\cP(k),M))\to
C(R\#_\cP\Gamma/\cP(k),M\#_\cP \Gamma).
\]
For $x$ as in \eqref{eq:eltens}, we put
\begin{equation}\label{eq:maphi}
\phi([x])=
[a_0\# f_1\otimes f_1^\dagger(a_1)\# f_2\otimes \dots\otimes
(f_1\cdots f_n)^\dagger(a_n)\#(f_1\cdots f_n)^\dagger].
\end{equation}
Here $[]$ denotes congruence class.

\begin{prop}\label{prop:phiso}
The assignment \eqref{eq:maphi} gives a simplicial isomorphism
\[\phi:\diag(\perp(\Gamma/\cP, C(R/\cP(k),M)))\iso C(R\#_\cP\Gamma/\cP(k),M\#_\cP \Gamma).\]
In particular, we have a quasi-isomorphism
\[
\bH(\Gamma/\cP,HH(R/\cP(k),M))\weq HH(R\#_\cP\Gamma/\cP(k),M\#_\cP \Gamma).
\]
\end{prop}
\begin{proof} First of all, we must check that \eqref{eq:maphi} gives a well-defined  simplicial homomorphism. To do this,
one checks first that formula \eqref{eq:maphi} defines a simplicial homomorphism
\[
\hat{\phi}:\diag(\perp(\Z[\emb], C(R,M)))\to C(R\#\emb,M\#\emb).
\]
Then one observes that it passes down to the quotient, inducing a
map $\phi:\diag(\perp(\Gamma/\cP, C(R/\cP(k),M)))\to
C(R\#_\cP\Gamma/\cP(k),M\#_\cP \Gamma)$. Next note that the image of
$\hat{\phi}$ is contained in the simplicial subgroup $$S\subset
C(R\#\emb,M\#\emb)$$ given in dimension $n$ by
\[
S_n=\mspan\{[a_0\#f_0\otimes\dots\otimes a_n\#f_n]: f_i\in \emb,\ \
a_i\in R,\ \ f_0\cdots f_n\in 2^\N\}.
\]
To prove that $\phi$ is surjective, we must show that $$S\to
C(R\#_\cP\Gamma/\cP(k),M\#_\cP \Gamma)$$ is surjective. Any element of
$C(R\#_\cP\Gamma/\cP(k),M\#_\cP \Gamma)$ can be written as a linear
combination of classes of elementary tensors of the form
\begin{equation}\label{elemtensor}
y=a_0\#f_0\otimes\dots\otimes a_n\#f_n,
\end{equation}
such that the following conditions are satisfied for $0\le i\le n-1$ and $0\le j\le n$:
\begin{equation}\label{condistensor}
\epsilon_r(f_i)=\epsilon_l(f_{i+1}),\quad
\epsilon_r(f_n)=\epsilon_l(f_0)\quad a_j=a_j\epsilon_l(f_j).
\end{equation}
Let $f=f_0\cdots f_n$; then $\dom(f)=\ran(f)=\ran(f_0)=\dom(f_n)$.
Let
\[
\N\supset A=\{x\in\dom(f):f(x)=x\}.
\]
If $A=\dom(f)$ then $f\in 2^\N$, and thus the element
\eqref{elemtensor} belongs to $S$. Otherwise, by Zorn's Lemma, there
exists $\emptyset\neq B\subset\dom(f)$ maximal with the property
that $f(B)\cap B=\emptyset$. Clearly $A\cap B=\emptyset$; let
$C=\dom(f)\backslash(A\sqcup B)$. Then $f(B)\subset C$, $f(C)\subset
B$, and $p_{\dom(f)}=p_A+p_B+p_C$. Hence we have
\[
[y]=[p_{\dom(f)}yp_{\dom(f)}]=[p_Ayp_A]=[a_0\#
g_0\otimes\dots\otimes a_n\#g_n],
\]
for $g_n=(f_n)_{|A}$ and $g_i=(f_i)_{|f_{i+1}\cdots f_n(A)}$ $(0\le
i\le n-1)$. In particular $g_0\cdots g_n=p_A$. Thus $\phi$ is
surjective. To prove it is injective, define a map
\[
\psi:C(R\#_\cP\Gamma/\cP(k),M\#_\cP \Gamma)\to \diag(\perp(\Gamma/\cP, C(R/\cP(k),M)))
\]
as follows. For $y$ as in \eqref{elemtensor} satisfying the
conditions \eqref{condistensor} and such that $f_0\cdots f_n\in
2^\N$, put
\[
\psi([y])= [a_0\otimes f_0(a_1)\otimes\dots\otimes (f_0\cdots
f_{n-1})(a_n)\otimes f_0\otimes\dots\otimes f_{n-1}].
\]
One checks that $\psi$ is well-defined and that $\psi\phi=id$.
\end{proof}

\begin{coro}\label{cor:hh0}
Assume that $R$ is commutative and that $M$ is a central
$R$-bimodule. Then
\[
HH_0(R\#_\cP\Gamma,M\#_\cP\Gamma)=M_\cE.
\]
\end{coro}
\begin{proof}
By Proposition \ref{prop:phiso},
\[HH_0(R\#_\cP\Gamma,M\#_\cP\Gamma)=H_0(\Gamma/\cP,HH_0(R,M)).\]
By our assumptions on $R$ and $M$, $HH_0(R,M)=M$. Finally we have
$H_0(\Gamma/\cP,M)=M_\cE$, by Proposition \ref{prop:ce=estar}.
\end{proof}

\subsection{Comparing the \topdf{$0^{th}$-homology}{zero homology} of \topdf{$(\Gami,I_S)$}{Gamic, IS} and
that of \topdf{$(\cB:J_S)$}{B:J}}

\begin{prop}\label{prop:hh0m}
Let $S\triqui\elli$ be a symmetric ideal and let
$J_S\triqui\cB=\cB(\ell^2)$ be the corresponding ideal of bounded
operators in $\ell^2$. Then the inclusion $\Gami\subset\cB$ induces
an isomorphism
\[
HH_0(\Gami,I_S)\iso HH_0(\cB,J_S).
\]
\end{prop}
\begin{proof}
By Proposition \ref{prop:cpg} Corollary \ref{cor:hh0}, the inclusion
$\diag:S\to I_S$ descends to a bijection
\begin{equation}\label{map:stohh0}
S_\cE\iso HH_0(\Gami,I_S).
\end{equation}
By \cite{kenetal2}*{Theorem 5.12} the composite of
\eqref{map:stohh0} with the map induced by the inclusion $I_S\subset
J_S$ is an isomorphism.
\end{proof}

\begin{coro}\label{coro:hc0}
The map $HC_0(\Gami:I_S)\to HC_0(\cB:J_S)$ is an isomorphism.
\end{coro}
\begin{proof}
It follows from Proposition \ref{prop:hh0m} and the fact that, if
$R$ is a unital ring and $I\triqui R$ is an ideal then
$$HH_0(R:I)=HC_0(R:I)=I/[R,I].$$
\end{proof}

\begin{lem}\label{lem:previo}
Let $p>0$. Then:
\begin{gather*}
HC_0(\Gami:I_{\ell^{p+}})=\left\{\begin{matrix}\C & p<1\\ 0 & p\ge 1\end{matrix}\right.\\
HC_0(\Gami:I_{\ell^{p-}})=\left\{\begin{matrix}\C & p\le 1\\ 0&
p>1\end{matrix}\right.\\
HC_0(\Gami:I_{\ellp})=\left\{
\begin{matrix}\C & p<1\\ \C\oplus \V & p=1 \\0& p>1.\end{matrix}
\right.\\
\end{gather*}
Here $\V$ is a $\C$-vector space of uncountable dimension.
\end{lem}
\begin{proof} It follows from Corollary \ref{coro:hc0} and
\cite{wodk}*{pages 492-493}.
\end{proof}

\subsection{Cyclic homology of \topdf{$R\#_\cP\Gamma$}{Crossed prod}}\label{subsec:HCRbundle}

\goodbreak

Now we go back to the general situation of Subsection \ref{subsec:HRbundle}. So $k$ is a field and  $R$ is an $\emb$-bundle over $k$.
Let $M$ be a right $\Gamma$-module. Consider the simplicial module
$\perp (\Gamma/\cP,M)$. Every element of $\perp_n(\Gamma/\cP,M)$ can
be written as a sum of elementary tensors
\[
x=m\otimes f_1\otimes\dots\otimes f_n
\]
with $m\in M$, $f_i\in\emb$, and $\dom(f_i)=\ran(f_{i+1})$ $(i<n)$.
For $x$ as above, put
\begin{equation}\label{map:tau}
\tau_n(x)=(-1)^n m(f_1\cdots f_n)\otimes (f_1\cdots
f_n)^\dagger\otimes f_1\otimes\dots\otimes f_{n-1}.
\end{equation}
One checks that the assignment \eqref{map:tau} gives a well-defined
endomorphism of $\perp_n(\Gamma/\cP,M)$, and that the cyclic
identities \cite{lod}*{2.5.1.1} hold. Thus the simplicial
($k$-)module $\perp(\Gamma/\cP,M)$, equipped with the cyclic
operators $\tau_n$ ($n\ge 0$), is a \emph{cyclic module}. In general
if $\cC$ is any cyclic module, then we can equip $\cC$ with a map
$B:\cC\to \cC[+1]$ called the Connes' operator, which, together with
the usual boundary $b:\cC\to \cC[-1]$ given by the alternating sum
of the face maps, satisfy $b^2=B^2=[b,B]=0$. When $\cC=\perp (\Gamma/\cP,M)$, we write
$\partial$ and $\cB$ for the operators $b$ and $B$.
The \emph{Hochschild
complex} of a cyclic module $\cC$ is $HH(\cC)=(\cC,b)$. The \emph{cyclic} and
\emph{negative cyclic} complexes are the complexes given in
dimension $n$ by $HC(\cC)_n=\bigoplus_{m\ge 0} \cC_{n-2m}$ and
$HN(\cC)_n=\prod_{m\ge 0} \cC_{n+2m}$; they are equipped with the
boundary $b+B$. Observe that $HC(\cC)$ is also equipped with a chain
map $S:HC(\cC)\to HC(\cC)[-2]$ defined by the obvious projections
$HC(\cC)_n\to HC(\cC)_{n-2}$. If $C$ is another chain complex
equipped with a chain map $S:C\to C[-2]$, then by a \emph{map of
$S$-complexes} $C\to HC(\cC)$ we understand a chain map which
commutes with $S$. 
\begin{prop}\label{prop:perturb}
There is a natural quasi-isomorphism of $S$-complexes $(HC(\perp
(\Gamma/\cP,M)),\partial)\to (HC(\perp
(\Gamma/\cP,M)),\partial+\cB)$.
\end{prop}
\begin{proof}
View $\cC=\perp(\Gamma/\cP,M)$ as a cyclic module. Consider the
projection
\[
\pi:HN(\cC)_n=\prod_{m\ge 0}\cC_{n+2m}\to \cC_n=HH(\cC)_n.
\]
Observe that $\pi(b+\cB)=b\pi$. Proceed as in \cite{agree}*{\S 3.1}
to define a chain map $\Upsilon:HH(\cC)\to HN(\cC)$ such that
$\pi\Upsilon=1$. We have a chain map $\theta^{n}:HN(\cC)\to
HC(\cC)[2n]$ ($n\ge 0$) given by the composite
\begin{multline*}
\theta^n: HN(\cC)_p=\prod_{m\ge 0}\cC_{p+2m}\twoheadrightarrow
\bigoplus_{m=0}^n\cC_{p+2m}\\
\subset \bigoplus_{q\ge 0}\cC_{p+2(n-q)}=HC(\cC)_{p+2n}.
\end{multline*}
The map of the proposition is
\[
\sum_{n=0}^\infty\theta^n\Upsilon:(HC(\cC),\partial)=\bigoplus_{n\ge
0}HH(\cC)[-2n]\to (HC(\cC),b+\cB).
\]
\end{proof}

\begin{thm}\label{thm:hccross}
Let $k$ be a field and $R$ an $\emb$-bundle over $k$. There is a natural zig-zag of
quasi-isomorphisms
\[
\bH(\Gamma/\cP,HC(R/\cP(k)))\weq HC(R\#_\cP\Gamma/k).
\]
\end{thm}
\begin{proof}
Consider the bicyclic module
\begin{equation}\label{bicyc}
\cC_{*,*}: ([m],[n])\mapsto\perp_m (\Gamma/\cP,C_n(R/\cP(k))).
\end{equation}
It follows from Proposition \ref{prop:perturb} that the total cyclic
complex \[T=(HC(\cC_{*,*}),b+\partial+B+\cB)\] is quasi-isomorphic
to
\[
(HC(\cC_{*,*}),b+\partial+B),
\]
which in turn is a model for $\bH(\Gamma/\cP,HC(R/\cP(k)))$. By the
cylindrical version of the Eilenberg-Zilber theorem
(\cite{khal}*{Theorem 3.1}), the complex $T$ is $S$-equivalent to
the $HC$-complex of the diagonal $\Delta$ of \eqref{bicyc}. By
Proposition \eqref{prop:phiso}, the map \eqref{eq:maphi} is an
isomorphism of simplicial modules $\Delta\iso C(R\#_\cP\Gamma/\cP(k))$; one
checks that it is actually an isomorphism of cyclic modules.
Finally, by Example \ref{exa:hhp}, the projection $C(R\#_\cP\Gamma/k)\to
C(R\#_\cP\Gamma/\cP(k))$ induces a quasi-isomorphism 
\begin{equation}\label{map:quisintro}
HC(R\#_\cP\Gamma/k)\to
HC(R\#_\cP\Gamma/\cP(k)).
\end{equation}
\end{proof}

\begin{coro}\label{coro:hccross}
Let $\fA$ be a bornological algebra and $S\triqui\elli$ a symmetric ideal. Then
\[
HC_*(\Gami(\fA):I_{S(\fA)})=\bH_*(\Gamma/\cP:HC((\elli(\fA):S(\fA))/\cP)).
\]
\end{coro}
\begin{proof} By Proposition \ref{prop:cpg}, we have $\Gami(\fA)=\elli(\fA)\#_\cP\Gamma$ and $I_{S(\fA)}=S(\fA)\#_\cP\Gamma$.
Now apply Theorem \ref{thm:hccross} and take fibers.  
\end{proof}

\goodbreak

\subsection{Hodge decomposition}\label{subsec:hodge}
If $R$ is a commutative $\Q$-algebra, then there are defined Adams
operations on $C(R)$, and we have an eigenspace decomposition
\cite{lod}*{Theorems 4.5.10 and 4.6.7}
\begin{equation}\label{hodge}
C(R)=\bigoplus_{p\ge 0} C^{(p)}(R),
\end{equation}
called the \emph{Hodge decomposition}. We have $C^{(p)}_n=0$ for
$n<p$ and each $C^{(p)}$ is a graded $R$-submodule, closed under the
Hochschild boundary map $b$. Thus, if $M$ is a central $R$-bimodule,
for $HH^{(p)}(R,M)=M\otimes_R (C^{(p)}(R),b)$ we have
\[
HH_n(R,M)=\bigoplus_{p\ge 0}^nHH^{(p)}_n(R,M).
\]
The Connes operator $B$ sends $C^{(p)}$ to $C^{(p+1)}$. Thus, we
have a direct sum decomposition of the cyclic complex
\[
HC(R)=\bigoplus_{p=0}^\infty HC^{(p)}(R)
\]
where
\[
HC^{(p)}(R)_n=\bigoplus_{p\ge 0}^nC^{(n-p)}_{n-2p}(R).
\]
Hence for $HC_*^{(p)}(R)=H_*(HC^{(p)}(R))$,
\[
HC_n(R)=\bigoplus_{p=0}^nHC_n^{(p)}(R).
\]
Let $(\Omega^*_{R},d)$ be the $DGA$ of (absolute) K\"ahler differential forms.
There is a natural map of mixed complexes
\begin{gather}
\mu:(C(R),b,B)\to (\Omega_{R},0,d)\nonumber\\
\mu(x_0\otimes\dots\otimes x_n)=(1/n!)x_0dx_1\land\cdots\land
dx_n.\label{map:mu}
\end{gather}
Let $M$ be a central $R$-bimodule; the map $\mu$ induces
isomorphisms
\begin{gather}
HH_n^{(n)}(R,M)=M\otimes_R\Omega^n_{R}\label{hhnn}\\
\mbox{and }
HC_n^{(n)}(R)=\Omega^n_{R}/d(\Omega^{n-1}_{R}).\label{hcnn}
\end{gather}

We say that $R$ is \emph{homologically smooth} if \eqref{map:mu} is
a quasi-isomorphism.

\begin{rem}
If $R$ happens to also be an algebra over $\cP$, then the Hodge decomposition
above induces a similar decomposition on $HH(R/\cP,M)$ and $HC(R/\cP)$,
so that $HH^{(p)}(R,M)\to HH^{(p)}(R/\cP,M)$ and $HH^{(p)}(R,M)\to HH^{(p)}(R/\cP)$
are quasi-isomorphisms. Moreover $\Omega_R\to \Omega_{R/\cP}$ is an isomorphism.
\end{rem}

\begin{exa}\label{exa:cxsmooth}
Let $R$ be a unital commutative complex $C^*$-algebra over $\C$. It
was proved in \cite{galgtop}*{Thm. 8.2.6} that $R$, regarded as a
$\Q$-algebra, is homologically smooth. In particular this applies
when $R=\ell^\infty$. Moreover, by \cite{galgtop}*{proof of Prop.
5.2.2}, $\elli$ is a filtering colimit of smooth $\C$-algebras. It
follows that $\Omega^n_{\elli}$ is a flat $\elli$-module for every
$n$. Hence
\[
HH_n(\elli, M)=M\otimes_{\elli}\Omega^n_{\elli}
\]
for every central bimodule $M$.
\end{exa}

Now assume that the commutative $\Q$-algebra $R$ is an $\emb$-bundle.
Then by Proposition \ref{prop:phiso}, Theorem \ref{thm:hccross}, and
naturality of the Hodge decomposition, we have quasi-isomorphisms
\begin{gather}
HH(R\#_\cP\Gamma,M\#_\cP\Gamma)\weq\bigoplus_{p\ge
0}\bH(\Gamma/\cP,HH^{(p)}(R/\cP,M))\\
\mbox{and }HC(R\#_\cP\Gamma)\weq \bigoplus_{p\ge
0}\bH(\Gamma/\cP,HC^{(p)}(R/\cP)).
\end{gather}

Put
\begin{gather}
HH_n^{(p)}(R\#_\cP\Gamma,M\#_\cP\Gamma)=\bH_n(\Gamma/\cP,HH^{(p)}(R/\cP,M)),\label{hodgenc}\\
HC_n^{(p)}(R\#_\cP\Gamma)=\bH_n(\Gamma/\cP,HC^{(p)}(R/\cP)).\nonumber
\end{gather}

We have decompositions
\begin{gather*}
HH_n(R\#_\cP\Gamma,M\#_\cP\Gamma)=\bigoplus_{p=0}^nHH_n^{(p)}(R\#_\cP\Gamma,M\#_\cP\Gamma),\\
HC_n(R\#_\cP\Gamma)=\bigoplus_{p=0}^nHC_n^{(p)}(R\#_\cP\Gamma).
\end{gather*}

If follows from \eqref{hhnn}, \eqref{hcnn}, and Proposition
\ref{prop:ce=estar} that
\begin{gather}
HH_n^{(n)}(R\#_\cP\Gamma,M\#_\cP\Gamma)=(M\otimes_R\Omega^n_{R})_\cE,\label{somehodgenc}\\
HC_n^{(n)}(R\#_\cP\Gamma)=(\Omega^n_{R}/d\Omega^{n-1}_{R})_\cE.\nonumber
\end{gather}

\section{The relative cyclic homology \topdf{$HC_*(\Gami(\fA):I_{S(\fA)})$}{HC(Gami:IS)}}\label{sec:wod}

\subsection{The Quillen spectral sequence}

Let $R$ be a unital $\Q$-algebra and $I\triqui R$ a two-sided ideal, flat both as a right and as a left ideal. Then
\[
I^{\otimes_R^n}\cong I^n.
\]
Using the isomorphism above and flatness again we see that if $P\weq
I$ is a projective bimodule resolution, then $Q=P^{\otimes_R^n}\weq
I^n$ is again a resolution. Hence modding out $Q$ by the commutator
subspace $[Q,R]$ we obtain a complex which computes $HH_*(R,I^n)$
and which has a natural action of $\Z/n\Z$ via permutation of
factors. Following Quillen \cite{qui}*{pp 210} we shall write
$HH_*(R,I^n)_\sigma$ for the coinvariants of this action. Quillen
introduced a first quadrant spectral sequence (see
\cite{qui}*{Proposition 2.16 and Theorem 4.3}),
\begin{equation}\label{specseq}
E^1_{p,q}=\left\{\begin{matrix}HC_q(R)&  p=0\\
HH_{q-p+1}(R,I^p)_\sigma& p\ge 1,\end{matrix}\right.
\end{equation}
which converges to $HC_{p+q}(R/I)$. For example, every ideal
$J\triqui\cB=\cB(\ell^2)$ of the algebra of bounded operators is
flat; M. Wodzicki has used this spectral sequence, together with the
results of \cite{kenetal2}, to study the relative cyclic homology
groups $HC_*(\cB:J)$. By Proposition \ref{prop:nflatgamma}, every
ideal of $\Gami$ is flat; by Proposition \ref{prop:flatellia} and
Examples \ref{exas:root}, the same is true of $I_{c_0(\fA)}$ and
$I_{\ell^{\infty-}(\fA)}$ for every unital Banach algebra $\fA$. In
this subsection we shall use Quillen's spectral sequence to study
the cyclic homology groups $HC_*(\Gami:I_S)$.
Proposition \ref{prop:tensor} below will play a role akin to that
played by \cite{wodk}*{Theorem 8} in the context of operator ideals.
Let $\fA$ and $\fB$ be Banach algebras, and let $\hotimes$ be the projective tensor product. We have maps
\begin{gather}
\Gamma\otimes\Gamma\to \Gamma(\N\times \N),\ \ U_f\otimes U_g\mapsto U_{f\times g},\label{map:tensogama}\\
\boxtimes:\elli(\fA)\otimes\elli(\fB)\to
\elli(\N\times\N,\fA\hotimes\fB),\ \
(\alpha\boxtimes\beta)_{m,n}=\alpha_n\hotimes\beta_m.\label{map:boxtimes}
\end{gather}
These two maps together induce
\begin{multline}\label{map:tensoab}
\Gami(\fA)\otimes\Gami(\fB)\to\\
\Gami(\N\times\N,\fA\hotimes\fB):=\elli(\N\times\N,\fA\hotimes\fB)\#_{\cP(\N\times\N)}\Gamma(\N\times\N).\nonumber
\end{multline}
We write $\Gami(\N\times\N)=\Gami(\N\times\N,\C)$. In particular we
have a map
\begin{equation}\label{map:tensocc}
\Gami\otimes\Gami\to \Gami(\N\times\N).
\end{equation}

\begin{prop}\label{prop:tensor}(cf.\cite{wodk}*{Theorem 8})

\smallskip

Let $S,T\triqui\elli$ be symmetric ideals, and let $\fB$ be a unital
Banach algebra. Assume that
\item[i)] The map \eqref{map:boxtimes} sends $S\otimes T\to T(\N\times\N)$.
\item[ii)] $S_{\cE}=0$.

Then $$HH_*(\Gami(\fB), I_{T(\fB)})=0.$$
\end{prop}
\begin{proof}
Proceeding as in the proof of \cite{hl}*{Proposition 7.3.4}, we
obtain a commutative diagram
\[
\xymatrix{\Gami\otimes\Gami(\fB)\ar[r]& M_2\Gami(\fB)\\
\Gami(\fB)\ar[u]^{E_{1,1}\otimes-}\ar[ur]&}
\]
By hypothesis i) this restricts to a commutative diagram
\[
\xymatrix{I_S\otimes I_{T(\fB)}\ar[r]& M_2 I_{T(\fB)}\\
I_{T(\fB)}\ar[u]^{E_{1,1}\otimes-}\ar[ur]&}
\]
Now use hypothesis ii), Morita invariance and the K\"unneth formula
for Hochschild homology (\cite{lod}*{Theorem 1.2.4} and
\cite{chubu}*{Proposition 9.4.1}), and induction, to conclude that
$HH_*(\Gami(\fA),I_{T(\fA)})=0$.
\end{proof}

We shall need the following result of Dykema, Figiel, Weiss and
Wodzicki, which follows by combining \cite{kenetal2}*{Theorem
5.11(ii) and Theorem 5.12}.

\begin{prop}\label{prop:citakenetal}(\cite{kenetal2})
Let $S\triqui\elli$ be a symmetric ideal and let $\omega=(1/n)_{n\ge
1}$ be the harmonic sequence. Then
\[
S_\cE=0\iff \omega\boxtimes S\subset S(\N\times\N).
\]
\end{prop}

\begin{prop}\label{prop:hcmismo}
\item[i)] $HC_*(\Gami:I_{c_0})=HC_*(\cB:J_{c_0})=0$.

\item[ii)] $HC_*(\Gami:I_{\ell^{\infty-}})=HC_*(\cB:J_{\ell^{\infty-}})=0$.

\item[iii)] Let $0<p<\infty$, $S\in\{\ell^p,\ell^{p-},\ell^{p+}\}$,
\begin{gather*}
m=\min\{n:HC_n(\Gami:I_{S})\neq 0\},\ \ \text{ and}\\
m'=\min \{n:HC_n(\cB:J_{S})\neq 0\}.
\end{gather*}
Then $m=m'$ and the map $HC_m(\Gami:I_S)\to HC_m(\cB:J_S)$ is an isomorphism.
\end{prop}
\begin{proof}
Consider the spectral sequence \eqref{specseq} in the cases
$R=\Gami,\cB$ and $I=I_S,J_S$ for each of the symmetric ideals $S$
of the proposition. We have $E^1_{0,*}=0$ since both $\Gami$ and
$\cB$ are rings with infinite sums (\cite{hl}*{\S5}). In both i) and ii), we have
$S^2=S$ and $\omega\boxtimes S\subset S(\N\times\N)$ whence
$E^1_{*,*}=0$, by Propositions \ref{prop:citakenetal} and
\ref{prop:tensor} and \cite{wodk}*{Theorem 8}. This gives i) and
ii). In each of the cases considered in part iii), we have
$S\boxtimes S\subset S(\N\times\N)$. Since $\omega\in\ell^p$ if and
only if $p>1$ and since $(\ell^p)^n=\ell^{p/n}$, we have
$HH_*(\Gami,I_{(\ell^p)^n})=HH_*(\cB,(\Ell^{p})^n)=0$ for $p/n>1$,
again by Propositions \ref{prop:citakenetal} and \ref{prop:tensor}
and \cite{wodk}*{Theorem 8}. The case $S=\ell^p$ follows from this
and from Corollary \ref{prop:hh0m}. The remaining cases follow
similarly.
\end{proof}

\begin{rem}\label{rem:preview}
Proposition \ref{prop:compu} below provides a more detailed computation of $HC_n(\Gami:I_S)$ for $S$ as in case iii) of
Proposition \ref{prop:hcmismo} above.
\end{rem}

\begin{thm}\label{thm:k=kh}
The comparison map $K_*(I_{S(\fA)})\to KH_*(I_{S(\fA)})$ is an
isomorphism in the following cases:
\item[i)] $S=c_0$ and $\fA$ is a $C^*$-algebra.
\item[ii)] $S=\ell^{\infty-}$ and $\fA$ is a unital Banach algebra.
\end{thm}
\begin{proof} By Proposition \ref{prop:tfp} and Examples \ref{exa:excicstar} and \ref{exa:exciunital}, $I_{S(\fA)}$
is $H$-unital in both cases. Hence by \eqref{intro:seqkth} it
suffices to show that $HC_*(\Gami(\fA):I_{S(\fA)})=0$. As explained
in the proof of Proposition \ref{prop:hcmismo}, Proposition
\ref{prop:citakenetal} implies that $S_\cE=0$. Hence if $\fA$ is
unital we are done by Propositions \ref{prop:flatellia} and
\ref{prop:tensor}; in particular, part ii) is proved. The nonunital
case of i) follows from the unital case using excision.
\end{proof}

\subsection{Computing \topdf{$HC^{(p)}(\Gami:I_S)$}{HC(Gami:IS)} in terms of differential forms}

Let $S\triqui \elli$ be an ideal. Consider the subcomplex
\begin{gather}
\cF_p(S)\subset \Omega_{\elli}\label{cF}\\
(\cF_p(S))^q=\left\{\begin{matrix}S^{p-q+1}\Omega^{q}_{\elli} & p\ge q\\
\Omega^q_{\elli}& q>p.\end{matrix}\right.\nonumber
\end{gather}
Write
\begin{gather}
 D^{(p)}(S)_q=(\Omega^{-q}_{\elli}/(\cF^{-q}_p(S))\\
L^{(p)}(S)_q=\cF^{-q}_{p-1}(S)/\cF^{-q}_{p}(S).
\end{gather}
Note $ L^{(p)}(S)$ and $D^{(p)}(S)$ are nonpositive chain complexes.

\begin{thm}\label{thm:cgg}
Let $S\triqui\elli$ be a symmetric ideal.
Then there are $\emb$-equivariant quasi-isomorphisms
\begin{gather*}
HH^{(p)}(\elli/S)\weq L^{(p)}(S)[p]\\
 HC^{(p)}(\elli/S)\weq
D^{(p)}(S)[p].
\end{gather*}
\end{thm}
\begin{proof}
Consider the skew-commutative graded algebra $\Lambda=\elli\oplus S$
with grading $\Lambda_0=\elli$, $\Lambda_1=S$. The inclusion
$S\subset \elli$ defines a homogeneous $\elli$-linear derivation
$\partial:\Lambda\to\Lambda[-1]$. Thus $\Lambda$ is a chain $DGA$,
and the projection $\elli\to \elli/S$ defines a quasi-isomorphism of
cyclic modules $C(\Lambda,\partial)\weq C(\elli/S)$. By
\cite{cgg}*{Thms. 2.6 and 3.3} and Proposition \ref{prop:ellippal},
there are quasi-isomorphisms $C(\Lambda,\partial)\weq
\bigoplus_pL^{(p)}(S)[p]$ and $\fB(\Lambda,\partial)\weq
\bigoplus_pD^{(p)}(S)[p]$; by \cite{vigue} they are compatible with
the Hodge decomposition. Finally, all these quasi-isomorphisms are
natural, and thus $\emb$-equivariant.
\end{proof}
\begin{thm}\label{thm:hcrel}
\begin{align*}
HC_*^{(p)}(\Gami:I_S)=&\bH_{*+p}(\Gamma/\cP,\cF_{(p)}(S))\\
HH_*^{(p)}(\Gami:I_S)=&\bH_{*+p+1}(\Gamma/\cP,L_{(p)}(S)).\\
\end{align*}
\end{thm}
\begin{proof}
It follows from \eqref{hodgenc} using Theorem \ref{thm:cgg} and the fact that
$\Gami$ is an infinite sum ring (\cite{hl}*{\S5}).
\end{proof}

\begin{coro}\label{cor:wodspec}
There is a first quadrant homological spectral sequence
\[
{}_pE^1_{m,n}=H_n(\Gamma/\cP,S^{m+1}\Omega^{p-m}_{\elli})\Rightarrow
HC_{m+n+p}^{(p)}(\Gami:I_S).
\]
\end{coro}
\begin{proof}
This is the spectral sequence associated to
$\bH(\Gamma/\cP,\cF_{(p)}(S))$. It is located in the first quadrant
because as $\Gami$ is an infinite sum ring,
\[HH^{(q)}_*(\Gami)=H_{*+q}(\Gamma/\cP,\Omega^q_{\elli})=0.\]
\end{proof}

\begin{coro}\label{cor:milspec}
\[
HC_n^{(n)}(\Gami:I_S)=(S\Omega^n_{\elli}/d(S^2\Omega^{n-1}_{\elli}))_\cE.
\]
\end{coro}
\begin{proof} It follows from inspection of the second term of the spectral sequence of Corollary
\ref{cor:wodspec}, by using the fact that $H_0(\Gamma/\cP,-)=(\ \
)_\cE$ is right exact.
\end{proof}

\subsection{The cases \topdf{$S=\ell^p,\ell^{p\pm}$}{schatten}}

\begin{lem}\label{lem:moritegama}
Let $S\triqui\elli$ be a symmetric ideal. Then the map
\[
C(\Gamma/\cP,S\Omega^p_{\elli})\to
C(\Gamma(\N\sqcup\N)/\cP(\N\sqcup\N),S(\N\sqcup\N)\Omega^p_{\elli(\N\sqcup\N)})
\]
induced by the inclusion $\N\subset \N\sqcup\N$ into the first copy,
is a quasi-isomor- phism.
\end{lem}
\begin{proof}
Recall from Corollary \ref{coro:flatelli} that every ideal of $\elli$ is flat, and from Example \ref{exa:cxsmooth} that $\Omega^p_{\elli}$ is a
flat $\elli$-module. It follows that the map
$S\otimes_{\elli}\Omega^p_{\elli}\to S\Omega^p_{\elli}$ is an
isomorphism for every ideal $S$. Now the proof is immediate from
\cite{hl}*{Lemma 7.3.1} and Lemma \ref{lem:moritesuple}.
\end{proof}

\begin{lem}\label{lem:vanishmap}
Let $0\ne S_1,S_2\subset\elli$ be symmetric ideals. Assume that $(S_1)_\cE=0$ and that the map
$\elli\otimes\elli\to \elli(\N\times\N)$ sends $S_1\otimes S_2\to S_2(\N\times\N)$. Then
$H_*(\Gamma/\cP,S_2\Omega^p_{\elli})=0$ $(p\ge 0)$.
\end{lem}
\begin{proof}
The proof follows using Lemma \ref{lem:moritegama} and the argument of the proof of Proposition \ref{prop:tensor}.
\end{proof}

Let $p\in\R$; the following notation is used in the proposition below.
\[
[p]=\max\{n\in\Z:n\le p\},\ \ \lf p\rf=\left\{\begin{matrix} p-1 & p\in\Z\\ [p] &
p\notin\Z.\end{matrix}\right.
\]

\begin{prop}\label{prop:compu}
\item[i)] Let $p>0$ and let $S_p$ be either $\ell^p$ or $\ell^{p-}$. Then
\begin{multline*}
HC_n^{(q)}(\Gami:I_{S_p})=\\
\left\{\begin{matrix}0 & n<q+\lf p\rf\\
(S_{(p/(\lf p\rf+1))}\Omega^{q-\lf p\rf}_{\elli}/d(S_{(p/(\lf p\rf+2))}\Omega^{q-\lf p\rf-1}_{\elli}))_\cE & n=q+\lf
p\rf.\end{matrix}\right.
\end{multline*}
In particular, the first nonzero group is
\[
HC_{2\lf p\rf}(\Gami:I_{S_p})=HC_{2\lf p\rf}^{\lf
p\rf}(\Gami:I_{S_p})=HC_0(\Gami:I_{S_{p/(\lf p\rf+1)}})
\]
which was computed in \ref{lem:previo}.
\item[ii)]
\begin{multline*}
HC_n^{(q)}(\Gami:I_{\ell^{p+}})=\\ \left\{\begin{matrix}0 & n<q+[p]\\
(\ell^{(p/([p]+1))+}\Omega^{q-[p]}_{\elli}/d(\ell^{(p/([p]+2))+}\Omega^{q-[p]-1}_{\elli}))_\cE& n=q+[p].
\end{matrix}\right.
\end{multline*}
In particular, the first nonzero group is
\[
HC_{2[p]}(\Gami:I_{\ell^{p+}})=HC_{2[p]}^{([p])}(\Gami:I_{\ell^{p+}})=HC_0(\Gami:I_{\ell^{(p/([p]+1))^+}})=\C
\]
\end{prop}
\begin{proof} This is a straightforward application of the spectral sequence of Corollary \ref{cor:wodspec} together with Lemma
\ref{lem:vanishmap} and Proposition \ref{prop:citakenetal}.
\end{proof}

\end{document}